\documentclass[BSc]{usydthesis}

\usepackage{amsfonts}

\usepackage{geometry}
\author{Andrew Rajchert\\
Supervisor: Dmitry Badziahin}
\title{On the Irrationality Exponents of Mahler Numbers}

\numberwithin{equation}{chapter}

\newtheorem{Definition}[equation]{Definition}
\newtheorem{Theorem}[equation]{Theorem}
\newtheorem{Proposition}[equation]{Proposition}
\newtheorem{Lemma}[equation]{Lemma}
\newtheorem{Corollary}[equation]{Corollary}

\theoremstyle{remark}

\newcommand{\N}{\mathbb{N}}
\newcommand{\R}{\mathbb{R}}
\newcommand{\Q}{\mathbb{Q}}
\newcommand{\Z}{\mathbb{Z}}
\newcommand{\C}{\mathbb{C}}

\newcommand{\rad}{\text{rad}}

\begin{document}    

\pagenumbering{roman}

\renewcommand{\Today}{December 2021}
\maketitle          
\tableofcontents    


%
%
%
%
%
%

\newpage\setcounter{page}{1}\pagenumbering{arabic}

\chapter{Introduction} 
An irrationality exponent, sometimes called an irrationality measure, is a way of measuring how well an irrational number can be approximated by rational numbers. As $\Q$ is dense in $\R$, it is certainly possible to find arbitrarily accurate rational approximations to any given real number, however the irrationality exponent also takes into account how "complicated" the rational approximation needs to be through the size of the denominator. Additionally, the irrationality exponent takes into account how well an irrational number may be consistently approximated, so infrequent accurate approximations are discarded.

\begin{Definition}
For a fixed $\alpha \in \R$, Consider the inequality
\begin{align}
\Big|\alpha - \frac{p}{q}\Big| < \frac{1}{q^{\mu}}, \label{expo}
\end{align}

where $p, q \in \Z$ and $\mu \in \R$. Consider the set of $\mu \in \R$ such that there are infinitely many integer solutions for $p$ and $q$ to (\ref{expo}). The irrationality exponent of $\alpha$, denoted by $\mu(\alpha)$, is defined as the supremum of this set.
\end{Definition}
It is already known, by Dirichlet's approximation theorem, that for any $\alpha \in \mathbb{R}$, we have $\mu(\alpha) \geq 2$ \cite{hardy}. Furthermore, for any $\mu \in [2, \infty]$, there exists an $\alpha \in \mathbb{R}$ such that $\mu(\alpha) = \mu$ \cite{sondow}.\\
\\
The irrationality exponent is a powerful and interesting tool in number theory and the study of transcendental numbers. One early application of the irrationality exponent is the definition of the Liouville numbers, the set of irrational numbers $\alpha \in \R\setminus \Q$ such that $\mu(\alpha) = \infty$. Not only was it shown that these numbers exist, but they are also relatively easy to construct and are transcendental, thus becoming the first known examples of transcendental numbers \cite{kempner}. Over 100 years later, Roth's theorem was proven, which states that that all irrational algebraic numbers have an irrationality exponent of 2, providing a clear way to prove some numbers are transcendental, a result for which Klaus Roth won the 1958 Fields medal for discovering \cite{roth}.\\
\\
While a interesting and unique tool, the irrationality exponent of any given irrational number can be extremely difficult to compute, and almost immediately leads to open problems in mathematics. For example, as of the time of writing, the best known bounds on $\mu(\pi)$ are $2 \leq \mu(\pi) \leq 7.1032...$ \cite{pi}. In this thesis, we aim to study the irrationality exponents of a certain class of numbers, called \textit{Mahler Numbers}, as these numbers have several properties that makes it feasible to compute the irrationality exponent, despite the fact that the Mahler numbers capture a great variety of real numbers. Before we define this, we give some more notation useful for the rest of the report.

\begin{Definition}
Fix $n \in \Z$, and let $(c_k)_{k \geq -n}$ be a sequence in $\Q$. We say
\begin{align*}
f(z) = \sum_{k=-n}^\infty c_{k}z^{-k},
\end{align*}
is a formal Laurent series in $\Q$. The set of all formal Laurent series in $\Q$ is denoted $\Q((z^{-1}))$.
\end{Definition}

For our purposes, we will only be dealing with formal Laurent series over $\Q$ as opposed to any other set. If treated as a function, this is a power series of $1/z$, with some added polynomial part, hence it has some radius of convergence $\rho \in [0, \infty]$ where $f(z)$ converges in $\mathbb{C}$ for any $|z| > \rho$ and diverges for $|z| < \rho$. Although there are choices of the $c_k$ coefficients such that the function does not converge anywhere in $\C$, this only presents an issue when evaluating $f(z)$, but we may still treat it as a power series.

\begin{Definition}
For $f \in \Q((z^{-1}))$, we say $f$ is a Mahler Function if there exists some integers $n \geq 1, d \geq 2$ and polynomials $P_0, P_1, ..., P_n, Q \in \Q[z]$ such that $P_0(z)P_n(z) \neq 0$ and
\begin{align}
\sum_{i = 0}^n P_i(z)f(z^{d^i}) = Q(z), \label{oldMahler}
\end{align}
for any $z$ inside the disc of convergence of $f$. For any integer $b$ in the disc of convergence of $f$, $f(b)$ is called a Mahler number.
\end{Definition}

We aim to study a subset of the Mahler functions. More specifically, given $f \in \Q((z^{-1}))$ satisfies the functional equation,
\begin{align}
f(z) = \frac{A(z)f(z^d) + C(z)}{B(z)},\label{equation}
\end{align}
where $A(z), B(z), C(z) \in \Q[z], A(z)B(z) \neq 0$, what can we say about the irrationality exponents of the Mahler numbers associated with $f$? Can they be computed in a (relatively) fast manner, and what values do they take? This question is motivated Adamczewski and Rivoal, who ask if the irrationality exponent of an automatic number is always rational \cite{adam}. We will not define what an automatic number is, but simply state that it forms a subset of the Mahler numbers \cite{badziahin}, and thus a significant step into solving their proposed problem is achieved by investigating Mahler numbers.\\
\\
For simplicity, unless explicitly stated otherwise, the term Mahler function is used to describe any Mahler function satisfying (\ref{equation}) for the rest of this thesis, and similarly with Mahler numbers. Note that it is not guaranteed that a Mahler number is irrational. It is possible that a function of the form $P(z)/Q(z), P, Q \in \Q[z]$ satisfies (\ref{equation}), leading to all associated Mahler numbers being rational. As these are of little interest in terms of the irrationality exponent, we will only focus on irrational Mahler numbers, and thus avoid situations where a rational function may satisfy (\ref{equation}).\\
\\
The case of $C(z) = 0$ is studied in depth by Badziahin, who finds the irrationality exponent of many Mahler numbers originating from (\ref{equation}) are rational, partially answering the question posed by Adamczewski and Rivoal. Furthermore, given a Mahler function with $C(z) = 0$ and some additional minor assumptions, there is a (relatively) simple method to calculate the irrationality exponent of these Mahler numbers exactly \cite{badziahin}.\\
\\
The purpose of this report is to produce results that are analogous to that of Badziahin in this slightly more generalised form, where $C(z)$ may be non-zero. Furthermore, we cover this in a more approachable manner, including relevant background and introductory information so the work can be understood by a wider range of readers.

\chapter{Continued Fractions for Formal Laurent Series}
An essential tool for this essay is the theory of continued fractions for formal Laurent series. Given a function $f \in \Q((z^{-1}))$, it is possible to produce a sequence of rational functions that approximate $f$ very well. If $f$ is a Mahler function, we can
evaluate the rational function approximations of $f$ at an integer, which will produce a sequence of rational approximations for the associated Mahler number. This provides a way to get a lower bound on what the irrationality exponent of the Mahler number may be.\\
\\
The study of continued fractions for formal Laurent series is based on the theory of continued fractions for real numbers, a more well known area with many parallels. Traditionally continued fractions for real numbers is used to study irrationality exponents, however the functional equation (\ref{equation}) provides a helpful path when using continued fractions for formal Laurent series instead, especially when the Mahler number itself may be challenging to compute.\\
\\
It may be helpful to have some background knowledge on continued fractions for real numbers to contextualise the results of this chapter, however it is certainly not required to understand them. Before we cover continued fractions, we first study the space of formal Laurent series and introduce it as a metric space.\\
\\
We can equip the set $\Q((z^{-1}))$ with a metric $d$ to form a metric space, where the metric is defined as

\begin{align*}
d(f(z), g(z)) := \begin{cases}
0 &\text{ if $f(z) = g(z)$}\\
2^{\deg(f(z) - g(z))} &\text{ otherwise.}
\end{cases}
\end{align*}
All required conditions for a metric are simple to verify with the possible exception of the triangle inequality, which we will check. If $f(z) = g(z), f(z) = h(z) \text{ or } g(z) = h(z)$, then clearly $d(f(z), h(z)) \leq d(f(z), g(z)) + d(g(z), h(z))$. Otherwise,

\begin{align*}
d(f(z), h(z)) &= 2^{\deg(f(z)-g(z) + g(z) -h(z))}\\
&\leq 2^{\max(\deg(f(z)-g(z)), \deg(g(z)-h(z)))}\\
&\leq 2^{\deg(f(z)-g(z))} + 2^{\deg(g(z)-h(z))} = d(f(z), g(z)) + d(g(z), h(z)).
\end{align*}
It can also be shown that this metric space is complete. Indeed, let $f_n(z)$ be a Cauchy sequence in $\Q((z^{-1}))$, and consider the $z^k$ terms for some $k \in \Z$ as $n$ varies. As $f_n$ is a Cauchy sequence, there exists a $n_0 \in \N$ where for all $m, n \geq n_0$, $d(f_n(z), f_m(z)) < 2^k$. This implies $\deg(f_{n_0}(z) - f_n(z)) < k$ for all $n \geq n_0$, and hence the $x^k$ coefficient is eventually constant.\\
\\
Set $f(z)$ to be the formal Laurent series such that for each $k \in \Z$, the $x^k$ coefficient is the constant eventually reached by the $x^k$ terms in the $f_n(z)$ sequence. For any $\epsilon > 0$, fix an integer $k$ such that $k < \log_2(\epsilon)$. By construction there exists a $n_0$ where $\deg(f_n(z) - f(z)) < k$ for all $n \geq n_0$ as the coefficients of $f_n$ will eventually match those of $f$, hence $d(f_n(x), f(x)) < \epsilon$ for all $n \geq n_0$. This immediately implies $f_n \to f$ with respect to this metric, and as $f(x) \in \Q((z^{-1}))$, we have completeness of the space.\\
\\
We can now introduce continued fractions. The rest of this chapter makes use of the proofs and concepts given by Alf van der Poorten \cite{poorten}, however we present it much more rigorously with added structure similar to that of Hardy and Wright's discussion of continued fractions for real numbers \cite{hardy}.

\begin{Definition}
Let $(a_n(z))_{n \geq 0}$ be a sequence of functions in $\Q((z^{-1}))$ such that $\deg(a_n(z)) \geq 1$ for $n \geq 1$. Then for any $N \geq 0$,
\begin{align*}
[a_0(z), a_1(z), a_2(z), ..., a_N(z)] := a_0(z) + \frac{1}{a_1(z) + \frac{1}{a_2(z) + \frac{1}{... + \frac{1}{a_N(z)}}}}.
\end{align*}
This is called a finite continued fraction. Assuming the right hand side converges in $\Q((z^{-1}))$ as $N \to \infty$, we may have an infinite continued fraction:

\begin{align*}
[a_0(z), a_1(z), a_2(z), ...] &:= \lim_{N \to \infty} a_0(z) + \frac{1}{a_1(z) + \frac{1}{a_2(z) + \frac{1}{... + \frac{1}{a_N(z)}}}}\\
&= a_0(z) + \frac{1}{a_1(z) + \frac{1}{a_2(z) + \frac{1}{...}}}.
\end{align*}
\end{Definition}

This expression is a little cumbersome, but thankfully we can simplify it with the following proposition.

\begin{Proposition}
\label{prop1}
Let $(a_n(z))_{n \geq 0}$ be a sequence in $\Q((z^{-1}))$ such that $\deg(a_n(z)) \geq 1$ for $n \geq 1$. Let $p_0(z) = a_0(z)$, $q_0(z) = 1$, $p_1(z) = a_1(z)a_0(z) + 1$, $q_1(z) = a_1(z)$ and for $n \geq 2$, $p_n(z) = a_n(z)p_{n-1}(z) + p_{n-2}(z)$, $q_n(z) = a_n(z)q_{n-1}(z) + q_{n-2}(z)$. Then for $n \geq 0$,
\begin{align*}
[a_0(z), a_1(z), ..., a_n(z)] = \frac{p_n(z)}{q_n(z)}
\end{align*}
\end{Proposition}

\begin{proof}
We verify this by induction. The result is obvious for $n = 0, 1$ and 2 by basic algebraic manipulations, so assume it holds for some fixed $n \geq 2$. We now try to prove it for $n+1$.\\
\\
Consider a slightly different sequence $(a'_k(z))_{k \geq 0}$ where $a'_k(z) = a_k(z)$ for $k \neq n$, and $a'_n(z) = a_n(z) + \frac{1}{a_{n+1}(z)}$. As $\deg(a_k(z)) \geq 1$ for $k \geq 1$, we have $\deg(a'_k(z)) \geq 1$ for all $k \geq 1$ as well, thus the continued fraction involving these terms is well defined.\\
\\
Let $p'_k(z)$ and $q'_k(z)$ be defined in a similar manner to $p_k(z)$ and $q_k(z)$, except based off the sequence $(a'_k(z))_{k \geq 0}$. As $p'_{n-1}(z), p'_{n-2}(z), q'_{n-1}(z)$ and $q'_{n-2}(z)$ depend only on $a'_k(z)$ for $k \leq n-1$, which are identical to the $a_k(z)$ for $k \leq n-1$, we must have $p'_{n-1}(z) = p_{n-1}(z)$ and similarly for $p'_{n-2}(z), q'_{n-1}(z)$ and $q'_{n-2}(z)$. We can apply our inductive hypothesis on the continued fraction of $a'_k(z)$ to find

\begin{align*}
[a_0(z), a_1(z), ..., a_n(z), a_{n+1}(z)] &= [a_0(z), a_1(z), ..., a_n(z) + \frac{1 }{a_{n+1}(z)}]\\
&= [a'_0(z), a'_1(z), ..., a'_n(z)]\\
&= \frac{p'_n(z)}{q'_n(z)}\\
&= \frac{a'_n(z)p'_{n-1}(z) + p'_{n-2}(z)}{a'_n(z)p'_{n-1}(z) + p'_{n-2}(z)}\\
&= \frac{(a_n(z) + \frac{1}{a_{n+1}(z)})p_{n-1}(z) + p_{n-2}(z)}{(a_n(z) + \frac{1}{a_{n+1}(z)})q_{n-1}(z) + q_{n-2}(z)}\\
&= \frac{a_{n+1}(z)(a_n(z)p_{n-1}(z) + p_{n-2}(z)) + p_{n-1}(z)}{a_{n+1}(z)(a_n(z)q_{n-1}(z) + q_{n-2}(z)) + q_{n-1}(z)}\\
&= \frac{a_{n+1}(z)p_n(z) + p_{n-1}(z)}{a_{n+1}(z)q_n(z) + q_{n-1}(z)} = \frac{p_{n+1}(z)}{q_{n+1}(z)}.
\end{align*}
This completes the inductive step, hence the proposition is true for all $n \geq 0$.\\
\\
\end{proof}

The elements of the sequence $(p_n(z)/q_n(z))_{n \geq 0}$ are called \textit{convergents}, and are much easier to deal with than the continued fraction itself. As the name suggests, $p_n(z)/q_n(z)$ converges in $\Q((z^{-1}))$ as $n \to \infty$, which we will show momentarily.\\
\\
Also note that if $(a_n(z))_{n \geq 0}$ are all polynomials in $\Q[z]$, we have $p_n(z)$ and $q_n(z)$ being polynomials and hence $p_n(z)/q_n(z)$ is a rational function, thus we have a sequence of rational functions converging to some function in $\Q((z^{-1}))$, exactly what we desire in our overarching plan of approximating Mahler functions. 

\begin{Proposition}
\label{prop2}
Let $(a_n(z))_{n \geq 0}$ be a sequence of functions in $\Q((z^{-1}))$ such that $\deg(a_n(z)) \geq 1$ for $n \geq 1$, and let $(p_n(z)/q_n(z))_{n \geq 1}$ be the convergents formed by the $(a_n(z))_{n \geq 0}$ sequence. Then $p_n(z)/q_n(z)$ converges to some $f \in \Q((z^{-1}))$ as $n \to \infty$.
\end{Proposition}

\begin{proof}
As $\Q((z^{-1}))$ is a complete metric space, it is enough to show that the sequence of convergents is a Cauchy sequence. First note,
\begin{align*}
&p_n(z)q_{n-1}(z) - p_{n-1}(z)q_n(z)\\
&= (a_n(z)p_{n-1}(z) + p_{n-2}(z))q_{n-1}(z) - p_{n-1}(z)(a_n(z)q_{n-1}(z) + q_{n-2}(z))\\
&= (-1)(p_{n-1}(z)q_{n-2}(z) - p_{n-2}(z)q_{n-1}(z)).\\
\end{align*}
Also observe $p_1(z)q_0(z) - p_0(z)q_1(z) = (a_0(z)a_1(z)+1) - a_0(z)a_1(z) = 1$, thus with a simple induction it follows that

\begin{align}
p_n(z)q_{n-1}(z) - p_{n-1}(z)q_n(z) = (-1)^{n-1}. \label{cor1}
\end{align}
Using this, for any fixed $n > 0$,

\begin{align*}
\frac{p_n(z)}{q_n(z)} &= \frac{p_{n-1}(z)}{q_{n-1}(z)} + \frac{(-1)^{n-1}}{q_{n-1}(z)q_n(z)}\\
&= \frac{p_{n-2}(z)}{q_{n-2}(z)} + \frac{(-1)^{n-2}}{q_{n-2}(z)q_{n-1}(z)} + \frac{(-1)^{n-1}}{q_{n-1}(z)q_n(z)}\\
&\vdots\\
&= \frac{p_0(z)}{q_0(z)} + \sum_{k=0}^{n-1} \frac{(-1)^{k}}{q_{k}(z)q_{k+1}(z)}\\
&= a_0(z) + \sum_{k=0}^{n-1} \frac{(-1)^{k}}{q_{k}(z)q_{k+1}(z)}.\\
\end{align*}
Also observe that the degrees of the $q_n(z)$ are strictly increasing. Indeed, $\deg(q_1(z)) \geq 1 >  0 = \deg(q_0(z))$. Assuming $\deg(q_k(z)) > \deg(q_{k-1}(z))$ for some $k \geq 1$, we also have

\begin{align*}
\deg(q_{k+1}(z)) &= \deg(a_{k+1}(z)q_k(z) + q_{k-1}(z))\\
&= \deg(a_{k+1}(z)q_k(z)) > \deg(q_k(z)).
\end{align*}
It follows by induction that the degree of $q_n(z)$ is strictly increasing over $n$, and thus the degrees of the terms of the sum

\begin{align*}
    \sum_{k=0}^{n-1} \frac{(-1)^{k}}{q_{k}(z)q_{k+1}(z)},
\end{align*}
are strictly decreasing over $k$. Now showing that $p_n(z)/q_n(z)$ is a Cauchy sequence, fix $\epsilon > 0$. For any $1 \leq m < n$,
\begin{align*}
\frac{p_n(z)}{q_n(z)} - \frac{p_m(z)}{q_m(z)} &= \sum_{k = m}^{n-1} \frac{(-1)^{k}}{q_k(z)q_{k+1}(z)}.
\end{align*}
As the degrees of the terms are strictly decreasing, it follows that the degree of this sum is $-\deg(q_m(z))-\deg(q_{m+1}(z)) < -2\deg(q_m(z))$, hence by definition of the metric,
\begin{align*}
d\Big(\frac{p_n(z)}{q_n(z)}, \frac{p_m(z)}{q_m(z)}\Big) < 2^{-2\deg(q_m(z))}.
\end{align*}
As $\deg(q_m(z))$ is strictly increasing, simply set $n_0$ such that $-2\deg(q_{n_0}(z)) \leq \log_2(\epsilon)$, then for all $m, n \geq n_0$, 
\begin{align*}
d\Big(\frac{p_n(z)}{q_n(z)}, \frac{p_m(z)}{q_m(z)}\Big) < \epsilon.
\end{align*}
As this can be done for every $\epsilon > 0$, it follows the sequence of convergents is a Cauchy sequence and hence must converge in $\Q((z^{-1}))$ by completeness, with the limiting function expressed as
\begin{align}
\label{convseries}
f(z) := \lim_{n \to \infty} \frac{p_n(z)}{q_n(z)} = a_0(z) + \sum_{k=0}^\infty \frac{(-1)^{k}}{q_k(z)q_{k+1}(z)}.
\end{align}

\end{proof}

\begin{Corollary}
\label{corollay}
Let $p_n(z)/q_n(z)$ be a convergent of $f(z) \in \Q((z^{-1}))$ such that $p_n(z), q_n(z) \in \Q[z]$. Then $p_n(z)$ and $q_n(z)$ are coprime. That is, $\gcd(p_n(z), q_n(z))$ is a constant.
\end{Corollary}

\begin{proof}
From equation (\ref{cor1}), $p_n(z)q_{n-1}(z)-p_{n-1}(z)q_n(z) = (-1)^{n-1}$. Note however, for any pair of polynomials $P(z), Q(z) \in \Q[z]$, $p_n(z)P(z) + q_n(z)Q(z)$ is always a multiple of $\gcd(p_n(z), q_n(z))$. We immediately conclude that $\gcd(p_n(z), q_n(z))$ divides 1, hence it must be some constant.
\end{proof}
While this establishes that all continued fractions converge in $\Q((z^{-1}))$, the power of continued fractions is that given $f \in \Q((z^{-1}))$, it is very simple to find a sequence $(a_n(z))_{n \geq 0} \in \Q[z]$ such that $f = [a_0(z), a_1(z), ...]$.\\
\\
For $f \in \Q((z^{-1}))$, let $\lfloor f \rfloor$ be the polynomial part of $f$. That is, if $f(z) = \sum_{k=-n}^\infty c_{-k}z^{-k}$, then $\lfloor f(z) \rfloor = \sum_{k=0}^n c_kz^k$ if $n \geq 0$, and $\lfloor f(z) \rfloor = 0$ otherwise.

\begin{Theorem} [Continued Fraction Algorithm for Laurent Series]

Assume $f(z) \in \Q((z^{-1}))$ is not a rational function, and define $a_0(z) = \lfloor f(z) \rfloor$, $f_1(z) = (f(z) - a_0(z))^{-1}$. For $n \geq 1$, inductively define the sequences,
\begin{align*}
a_n(z) &= \lfloor f_n(z) \rfloor\\
f_{n+1}(z) &=  (f_n(z) - a_n(z))^{-1}.
\end{align*}
Then $f_n(z)$ is well defined for all $n \geq 1$, and

\begin{align*}
[a_0(z), a_1(z), a_2(z), ...] = f(z).
\end{align*}
\end{Theorem}

\begin{proof}
First, we need to check $f_n(z)$ is well defined for all $n \geq 1$. This will be true if $f(z) \neq a_0(z)$ and $f_n(z) \neq a_n(z)$ for all $n \geq 1$, as then we never have the issue of dividing by zero. Clearly this first condition is satisfied as $f(z) = a_0(z) = \lfloor f(z) \rfloor$ immediately implies $f(z)$ is a rational function, a contradiction. Suppose there exists a $n \geq 1$ where $f_n(z) = a_n(z)$. Then by construction,
\begin{align*}
f(z) &= a_0(z) + \frac{1}{f_1(z)}\\
&= a_0(z) + \frac{1}{a_1(z) + \frac{1}{f_2(z)}}\\
&\vdots\\
&= [a_0(z), a_1(z), ..., f_n(z)]\\
&= [a_0(z), a_1(z), ..., a_n(z)].
\end{align*}
By definition, $a_k(z)$ is a polynomial for all $k \geq 1$, hence this implies $f(z)$ is a rational function, a contradiction. It follows that the $f_n(z)$ sequence is well defined for all $n \geq 1$. Additionally note that for $n \geq 1$, $\deg(a_n(z)) \geq \deg(f_n(z)) = -\deg(f_{n-1}(z) - \lfloor f_{n-1}(z)\rfloor) \geq 1$, hence the continued fraction is well defined.\\
\\
Now showing that the continued fraction converges to $f(z)$, let $(p_n(z)/q_n(z))_{n \geq 0}$ be the convergents of the continued fraction $[a_0(z), a_1(z), a_2(z), ...]$. We have just shown that for any $n \geq 0$, $f(z) = [a_0(z), a_1(z), ..., a_n(z), f_{n+1}(z)]$. Let $(p_k'(z)/q_k'(z))_{0 \leq k \leq n+1}$ be the convergents of this finite continued fraction. By the recursive definition from Proposition \ref{prop1}, $p_k(z) = p_k'(z)$ and $q_k(z) = q_k'(z)$ for $1 \leq k \leq n$, hence
\begin{align*}
f(z) &= \frac{p_{n+1}'(z)}{q_{n+1}'(z)}\\
&= \frac{f_{n+1}(z)p_n(z) + p_{n-1}(z)}{f_{n+1}(z)q_{n}(z) + q_{n-1}(z)}.
\end{align*}
It follows that

\begin{align*}
q_n(z)f(z)-p_n(z) &= q_n(z)\frac{f_{n+1}(z)p_{n}(z) + p_{n-1}(z)}{f_{n+1}(z)q_{n}(z) + q_{n-1}(z)} - p_n(z)\\
&= \frac{p_{n-1}(z)q_n(z) - p_n(z)q_{n-1}(z)}{f_{n+1}(z)q_{n}(z) + q_{n-1}(z)}\\
&= \frac{(-1)^n}{f_{n+1}(z)q_{n}(z) + q_{n-1}(z)},
\end{align*}
where for the last equality, we have applied (\ref{cor1}). Note that, $\deg(f_{n-1}(z) - \lfloor f_{n-1}(z) \rfloor) < 0$, hence by definition of the $a_n(z)$ and $f_n(z)$ sequences, $\deg(f_n(z)) \geq 1$ for $n \geq 1$. Recalling the degrees of $q_n(z)$ were strictly increasing, we must have
\begin{align*}
\deg(q_n(z)f(z)-p_n(z)) &= -\deg(f_{n+1}(z)q_{n}(z) + q_{n-1}(z)) \\
&< -\deg(q_n(z)).
\end{align*}
This implies that the degree of $f(z) - p_n(z)/q_n(z)$ must converge to $-\infty$ as $n \to \infty$, hence by definition of the metric on $\Q((z^{-1}))$, as $n \to \infty$
\begin{align*}
d\Big(f(z), \frac{p_n(z)}{q_n(z)}\Big) \to 0.
\end{align*}
It follows the convergents converge to $f(z)$, and hence

\begin{align*}
f(z) = \lim_{n \to \infty}\frac{p_n(z)}{q_n(z)} = [a_0(z), a_1(z), a_2(z), ...].
\end{align*}
\end{proof}

We say $(p_n(z)/q_n(z))_{n \geq 0}$ is the sequence of convergents of $f \in \Q((z^{-1}))$ if it is the sequence of convergents for the continued fraction given by the continued fraction algorithm applied to $f$. It should be noted that, although we assumed that $f(z)$ was not a rational function to guarantee that the algorithm will not terminate, this is not strictly required. If $f(z)$ is a rational function, then it can be shown that the algorithm will terminate once an exact representation of $f$ is reached, leading to a finite continued fraction expression of $f$. As rational functions are not very interesting in the context of Mahler numbers, we don't place as much emphasis on cases such as these.\\
\\
We conclude this chapter with two very useful properties about continued fractions for Laurent series which allows us to very quickly check if a rational function is a convergent of a given Laurent series, and also show that the convergents are very good approximations of $f(z)$ in regards to how large the degree of the denominator of the convergent is.

\begin{Lemma}
\label{lemma1}
Let $(p_n(z)/q_n(z))_{n \geq 0}$ be the sequence of convergents for a non-rational formal Laurent series $f(z) \in \Q((z^{-1}))$. Then $\deg(q_n(z)f(z)-p_n(z)) = -\deg(q_{n+1}(z))$.
\end{Lemma}

\begin{proof}
Using the series expansion from (\ref{convseries}),
\begin{align*}
f(z) - \frac{p_n(z)}{q_n(z)} &= a_0(z) + \sum_{k = 0}^\infty \frac{(-1)^{k}}{q_k(z)q_{k+1}(z)} - a_0(z) - \sum_{k = 0}^{n-1} \frac{(-1)^{k}}{q_k(z)q_{k+1}(z)}\\
&= \sum_{k = n}^\infty \frac{(-1)^{k}}{q_k(z)q_{k+1}(z)}.
\end{align*}
Thus by multiplying by $q_n(z)$ and using the fact that the degrees of $q_n(z)$ are strictly increasing,

\begin{align*}
\deg(q_n(z)f(z) - p_n(z)) &= \deg\Big( \frac{(-1)^{n}}{q_{n+1}(z)} + \sum_{k = n+1}^\infty \frac{(-1)^{k}q_n(z)}{q_k(z)q_{k+1}(z)}\Big)\\
&= \deg\Big( \frac{(-1)^{n}}{q_{n+1}(z)}\Big) = -\deg(q_{n+1}(z)).
\end{align*}
\end{proof}

An immediate corollary of this is that $\deg(q_n(z)f(z)-p_n(z))$ is strictly decreasing over $n$, and so the quality of the rational approximation $p_n(z)/q_n(z)$ is consistently improving in our metric space.

\begin{Theorem}[Criterion for Convergents of Laurent Series]
\label{criterion}
Let $p(z), q(z) \in \Q[z]$ be coprime and $f(z) \in \Q((z^{-1}))$. Then $\deg(q(z)f(z)-p(z)) < -\deg(q(z))$ if and only if $p(z)/q(z)$ is a convergent of $f$.
\end{Theorem}

\begin{proof}
If $p(z)/q(z) = p_n(z)/q_n(z)$ is a convergent of $f$, the statement is clear from Lemma \ref{lemma1}, as $\deg(q(z)f(z)-p(z)) = -\deg(q_{n+1}(z)) < -\deg(q(z))$. Now suppose that $p(z)/q(z)$ is not a convergent of $f$. Fix $n$ such that $\deg(q_{n-1}(z)) \leq \deg(q(z)) < \deg(q_n(z))$, and consider the system of equations for $a(z)$ and $b(z)$,
\begin{align*}
q(z) = a(z)q_{n-1}(z) + b(z)q_n(z)\\
p(z) = a(z)p_{n-1}(z) + b(z)p_n(z).    
\end{align*}
From (\ref{cor1}) the corresponding matrix has determinant $(-1)^{n-1}$, hence there exists unique polynomial solutions for $a(z)$ and $b(z)$. If either $a(z)$ or $b(z)$ are zero, then $p(z)/q(z)$ is a convergent, so this cannot be the case. Furthermore, if $\deg(a(z)q_{n-1}(z)) \neq \deg(b(z)q_n(z))$, we will then find $\deg(q(z)) \geq \deg(q_n(z))$, a contradiction to our choice of $n$. Thus,
\begin{align*}
\deg(a(z)) &= \deg(b(z)) + \deg(q_n(z)) - \deg(q_{n-1}(z))\\
&\geq \deg(q_n(z)) - \deg(q(z)).
\end{align*}
In addition, by using Lemma \ref{lemma1},

\begin{align*}
\deg(a(z)(q_{n-1}(z)f(z)-p_{n-1}(z))) &= \deg(a(z))-\deg(q_n(z))\\
&= \deg(b(z)) - \deg(q_{n-1}(z))\\
&> \deg(b(z)) - \deg(q_{n+1}(z))\\
&= \deg(b(z)(q_n(z)f(z)-p_n(z))).
\end{align*}
Putting this all together,

\begin{align*}
\deg(q(z)f(z)-p(z)) &= \deg(a(z)(q_{n-1}(z)f(z)-p_{n-1}(z)) + b(z)(q_n(z)f(z)-p_n(z)))\\
&= \deg(a(z)(q_{n-1}(z)f(z)-p_{n-1}(z)))\\
&= \deg(a(z))-\deg(q_n(z))\\
&\geq -\deg(q(z)).
\end{align*}
\end{proof}

This criterion gives $\deg(f(z) - \frac{p(z)}{q(z)}) < -2\deg(q(z))$ if and only if $p(z)/q(z)$ is a convergent of $f(z)$, suggesting that only the convergents achieve such a high quality of approximation relative to the degree of the denominator. This is somewhat analogous to the irrationality exponent where we seek good rational approximations relative to the size of the denominators of the approximations. It is this suggestion that makes looking at the convergents of formal Laurent series a powerful candidate when discussing irrationality exponents.

\chapter{Approximation Properties of Convergents}

Recall, we are investigating formal Laurent series $f \in \Q((z^{-1}))$ that satisfy
\begin{align}
f(z) &= \frac{A(z)f(z^d) +C(z)}{B(z)}, \label{Mahler}
\end{align}
where $A, B, C \in \Q[z]$, $A(z)B(z) \neq 0$. If $C(z) = 0$, then the equation reduces down to the case of Badziahin which has already been analysed in detail, so for the rest of this essay we make the assumption $C(z) \neq 0$. Although $A, B$ and $C$ can have non-integer coefficients, we will assume they have integer coefficients as we can simply multiply the numerator and denominator of (\ref{Mahler}) by a rational constant to adjust the coefficients.\\
\\
The goal of this chapter is to find an expression for the irrationality exponent of a Mahler number in terms of the sequence of convergents of $f$. While continued fractions for Laurent series give the best approximations of $f$ in $\Q((z^{-1}))$, this does not mean the approximation is accurate when evaluating them at integers, despite this holding intuitively. In the special case of functions satisfying (\ref{Mahler}) however, it is possible to produce countably many rational approximations for a Mahler number $f(b)$ from any given convergent of $f$. This allows us to form an upper and lower bound of the irrationality exponent $\mu(f(b))$, which happen to be equal, giving an exact expression for the irrationality exponent. As we are working to the same goal as Badziahin but in a slightly more generalised context, much of the working in this chapter follows the same line of logic as his, but applied to (\ref{Mahler}). \\
\\
To simplify the working in this section, we make use of some asymptotic notation.

\begin{Definition}
Let $(x_n)_{n \geq 0}, (y_n)_{n \geq 0}$ be sequences in $\R$. We say:
\begin{itemize}
    \item $x_n \ll y_n$ if there exists a constant $C > 0$ such that $|x_n| \leq C|y_n|$ for all $n \geq 0$. Similarly, $x_n \gg y_n$ if there exists a $c > 0$ such that $|x_n| \geq c|y_n|$ for all $n \geq 0$. \footnote{This is the same as Big-O and Big-$\Omega$ notation, however we denote the relationship with inequalities here since we repeatedly make use of the transitivity of this property.}\\
    \item $x_n = \Theta(y_n)$ if $x_n \ll y_n$ and $x_n \gg y_n$.\\
    \item $x_n = o(y_n)$ if for all $\epsilon > 0$, there exists a $n_0 \in \N$ such that $|x_n| < \epsilon |y_n|$ for all $n \geq n_0$.
\end{itemize}
\end{Definition}

\section{Lower bound for $\mu(f(b))$}

As before, let $(p_n(z)/q_n(z))_{n \geq 0}$ be the sequence of convergents for $f$. We can define another strictly increasing sequence  $(d_n)_{n \geq 0} \in \N$ such that $d_n = \deg(q_n(z))$. With this, we can give a lower bound on the irrationality exponent of some Mahler numbers.

\begin{Lemma}
\label{lowerbound}
Let $f(z) \in \Q((z^{-1}))$ be a Mahler function that satisfies (\ref{Mahler}) and is not a rational function. Let $b \in \Z, |b| \geq 2$ be inside the disc of convergence of $f$ such that $A(b^{d^t})B(b^{d^t}) \neq 0$ for all $t \in \N$. Then the irrationality exponent of $f(b)$ is bounded from below by
\begin{align*}
\mu(f(b)) \geq 1 + \limsup_{k \to \infty} \frac{d_{k+1}}{d_k}
\end{align*}
\end{Lemma}

\begin{proof}
By Lemma $\ref{lemma1}$, for any fixed $k \in \N$ we can write
\begin{align}
\label{cfexp}
q_k(z)f(z) - p_k(z) &= \sum_{i = d_{k+1}}^\infty c_{k, i}z^{-i}
\end{align}
where $c_{k, i} \in \R$ for all $k \geq 0$ and $i \geq d_{k+1}$, and $c_{k, d_{k+1}} \neq 0$ for all $k \geq 0$. Replacing $z$ with $z^d$ and using the Mahler functional equation (\ref{Mahler}),
\begin{align}
\label{Mahlerapply}
q_k(z^d)B(z)f(z)-(p_k(z^d)A(z)+q_k(z^d)C(z)) = A(z)\sum_{i = d_{k+1}}^\infty c_{k, i}z^{-di}
\end{align}
we can repeat this procedure an additional $m-1$ times to find
\begin{align}
\label{MahlerRepeated}
q_{k, m}(z)f(z)-p_{k, m}(z) = U_m(z)\sum_{i = d_{k+1}}^\infty c_{k, i}z^{-d^mi}
\end{align}
where 
\begin{align}
\label{qkmForm}
q_{k, m}(z) := q_k(z^{d^m})\prod_{t=0}^{m-1}B(z^{d^t}), \quad U_m(z) := \prod_{t=0}^{m-1}A(z^{d^t}).
\end{align}
and $p_{k, m}(z)$ satisfies a recursive formula given by

\begin{align*}
p_{k, m}(z) = p_{k, m-1}(z^d)A(z)+q_{k, m-1}(z^d)C(z), \quad p_{k, 1}(z) = p_k(z^d)A(z)+q_k(z^d)C(z).
\end{align*}
Although not obvious, it can be checked that the closed formula
\begin{align}
p_{k, m}(z) = p_k(z^{d^m})\prod_{t = 0}^{m-1} A(z^{d^t}) + q_k(z^{d^m}) \sum_{u=0}^{m-1} \Big[\Big(\prod_{t=0}^{u-1} A(z^{d^t})\Big) C(z^{d^u})\Big(\prod_{t=u+1}^{m-1} B(z^{d^t})\Big)\Big], \label{pkmForm}
\end{align}
satisfies this recursive relationship. Indeed, by relabeling the summation and product indexes, a loose inductive proof is as follows:

\begin{align*}
&p_{k, m}(z) = p_{k, m-1}(z^d)A(z) + q_{k, m-1}(z^d)C(z)\\
&= A(z)p_{k}(z^{d^{m-1}d})\prod_{t = 1}^{m-1} A(z^{d^t}) + A(z)q_{k}(z^{d^{m-1}d}) \sum_{u=0}^{m-2} \Big[\Big(\prod_{t=1}^{u} A(z^{d^t})\Big) C(z^{d^{u+1}})\Big(\prod_{t=u+2}^{m-1} B(z^{d^t})\Big)\Big]\\
&+ C(z)q_k(z^{d^{m-1}d})\prod_{t=1}^{m-1}B(z^{d^t})\\
&= p_{k}(z^{d^{m}})\prod_{t = 0}^{m-1} A(z^{d^t}) + q_k(z^{d^m})\sum_{u=0}^{m-1} \Big[\Big(\prod_{t=0}^{u-1} A(z^{d^t})\Big) C(z^{d^u})\Big(\prod_{t=u+1}^{m-1} B(z^{d^t})\Big)\Big].\\
\end{align*}
Strictly speaking, the above expression has products that are not well defined as the lower index is larger than the upper index, however as a convention for the rest of this essay, we take their value to be 1. By evaluating $q_{k, m}(z)$ and $p_{k, m}(z)$ at $z = b$, (\ref{MahlerRepeated})  suggests that $p_{k, m}(b)/q_{k, m}(b)$ may be a reasonable approximation of $f(b)$ for large $m$. Furthermore, as $A, B$ and $C$ have integer coefficients, $q_{k, m}(b)$ and $p_{k, m}(b)$ will be rational. We must now find out how large the right hand side of (\ref{MahlerRepeated}) and $q_{k, m}(b)$ grows over $m$ to find an expression for the irrationality exponent. This is achieved with the following Lemma.

\begin{Lemma}
\label{lemma1.5}
Let $b \in \R, |b| > 1$ be inside the disc of convergence of $f$. Assume that for all $t \in \N$, $A(b^{d^t})B(b^{d^t}) \neq 0$. Then for each $k \geq 0$,
\begin{align*}
|q_{k, m}(b)| = \Theta(\beta^m |b|^{d^m(\frac{r_b}{d-1} + d_k)}), \quad \text{and} \quad |q_{k, m}(b)f(b)-p_{k, m}(b)| = \Theta(\alpha^m|b|^{d^m(\frac{r_a}{d-1}-d_{k+1})})
\end{align*}
with asymptotic dependence on $m$, where $\deg(A(z)) = r_a$, $\deg(B(z)) = r_b$ and the leading coefficients of $A(z)$ and $B(z)$ given by $\alpha$ and $\beta$ respectively.
\end{Lemma}

\begin{proof}[Proof of Lemma \ref{lemma1.5}]
Using (\ref{cfexp}), for any $k \geq 0$,
\begin{align*}
z^{d_{k+1}}(q_k(z)f(z)-p_k(z)) = \sum_{i=0}^\infty c_{k, i+d_{k+1}}z^{-i}.
\end{align*}
As we let $z \to \infty$, the RHS tends to the non-zero constant $c_{k, d_{k+1}}$. As $|b| > 1$ and $d \geq 2$, $|b|^{d^m}$ tends to $\infty$ as $m \to \infty$. We substitute this for our value of $z$ to give
\begin{align*}
\Big|\sum_{i=d_{k+1}}^\infty c_{k, i}b^{-d^m i} \Big| = \Theta(|b|^{-d^m d_{k+1}}),
\end{align*}
with asymptotic dependence on $m$. Now note we can write
\begin{align*}
\prod_{t=0}^\infty \frac{A(z^{d^t})}{\alpha z^{r_ad^t}} = \prod_{t = 0}^\infty P(z^{-d^t}),
\end{align*}
where $P(z)$ is a polynomial with $P(0) = 1$. We claim this infinite product converges absolutely if $|z| > 1$. This is obvious if $P(z) = 1$, otherwise we write $P(z) = 1 + z^iQ(z)$ where $i \geq 1$ and $Q$ is a polynomial with a non-zero constant term. Using relatively standard properties of infinite products, the infinite product converges absolutely if
\begin{align*}
\sum_{t = 0}^\infty |z^{-id^t}Q(z^{-d^t})| < \infty.
\end{align*}
Applying the ratio test will immediately show this series converges when $|z| > 1$. Indeed,
\begin{align*}
\frac{z^{-id^{t+1}}Q(z^{-d^{t+1}})}{z^{-id^{t}}Q(z^{-d^{t}})} &= z^{id^t(1-d)}\frac{Q(z^{-d^{t+1}})}{Q(z^{-d^{t}})} \longrightarrow 0 \times \frac{Q(0)}{Q(0)} = 0.
\end{align*}
We also know $A(b^{d^t}) \neq 0$ for all integers $t \geq 0$ and $|b| > 1$, hence
\begin{align*}
\prod_{t=0}^{m-1} \frac{A(b^{d^t})}{\alpha b^{r_ad^t}},
\end{align*}
converges to a non-zero number as $m \to \infty$. It follows
\begin{align*}
&\Big|\prod_{t=0}^{m-1} \frac{A(b^{d^t})}{\alpha b^{r_ad^t}}\Big| = \Theta(1),\\
\implies &\Big|\prod_{t=0}^{m-1} A(b^{d^t}) \Big| = \Theta(|\alpha|^m |b|^{(1+d + ... + d^{m-1})r_a}) = \Theta(|\alpha|^m|b|^{\frac{d^m}{d-1}r_a}).
\end{align*}
An analogous argument applies to a similar product over $B(b^{d^t})$. Putting this together,
\begin{align*}
|q_{k, m}(b)| = \Big|q_k(b^{d^m}) \prod_{t=0}^{m-1} B(b^{d^t}) \Big| &= \Theta(|\beta|^m|b|^{d^m(\frac{r_b}{d-1}+d_k)})\\
|q_{k, m}(b)f(b)-p_{k, m}(b)| &= \Big|\prod_{t = 0}^{m-1}A(b^{d^t}) \sum_{i=d_{k+1}}^\infty c_{k, i}b^{-d^mi}\Big|\\
&= \Theta(|\alpha|^m |b|^{d^m(\frac{r_a}{d-1}-d_{k+1})})
\end{align*}
\end{proof}

Applying this Lemma and continuing the proof of Lemma \ref{lowerbound}, we can say there exists a sequence $(x_m)_{m \geq 1}$ and constants $c, C \in \R$ such that $0 < c < x_m < C$ for all $m$ and,
\begin{align*}
|q_{k, m}(b)| &= x_m |\beta|^m|b|^{d^m(\frac{r_b}{d-1}+d_k)}\\
&= |b|^{d^m(\frac{r_b}{d-1}+d_k) + m\ln|\beta|/\ln|b| + \ln(x_m)/\ln|b|}.\\
\end{align*}
A similar manipulation can be done for $|q_{k, m}(b)f(b) - p_{k, m}(b)|$. Noting that the $d^m$ term grows much faster than all other terms in the exponent, we can write
\begin{align}
|q_{k, m}(b)| &= b^{d^m(d_k+\frac{r_b}{d-1}+o(1))} \label{qkm}\\
|q_{k, m}(b)f(b)-p_{k, m}(b)| &=  b^{-d^m(d_{k+1}-\frac{r_a}{d-1}+o(1))},\label{pkm}
\end{align}
where the $o(1)$ terms converge to zero as $m \to \infty$.\\
\\
Fix $\epsilon > 0$ and set $k = k(\epsilon)$ such that 
\begin{align*}
 d_k > \frac{1}{\epsilon} \max\Big(\frac{r_a}{d-1}+1, \frac{r_b}{d-1}+1\Big). 
\end{align*}
This is possible as $d_k$ is monotone increasing to $\infty$. Now choose $m_0 = m_0(k)$ such that absolute value of the $o(1)$ terms are smaller than $1/2$ for all $m \geq m_0$.\\
\\
\underline{Claim:}
\begin{align*}
|q_{k, m}(b)|^{-\frac{d_{k+1}(1+\epsilon)}{d_k(1-\epsilon)}} < |q_{k, m}(b)f(b)-p_{k, m}(b)| < |q_{k, m}(b)|^{-\frac{d_{k+1}(1-\epsilon)}{d_k(1+\epsilon)}}.
\end{align*}
Supposing this claim is true, then by increasing $m$, we gain an infinite sequence of rational approximations to $f(b)$ with the quality of approximation in terms of the size of the denominator, exactly as desired for the irrationality exponent. As $\epsilon$ tends to zero, $k$ tends to $\infty$ as a function of $\epsilon$ and the claim suggests that
\begin{align*}
\mu(f(b)) \geq 1 + \limsup_{k \to \infty}\frac{d_{k+1}}{d_k}.
\end{align*}
This has not been completely justified, as we have only verified that $p_{k, m}(b)$ and $q_{k, m}(b)$ are rational rather than integers, as the definition of the irrationality exponent requires. Note however that by (\ref{qkmForm}) and (\ref{pkmForm}), we can verify that the denominator of $p_{k, m}(b)$ and $q_{k, m}(b)$ are independent of $m$, as $A, B$ and $C$ were assumed to have integer coefficients, and the denominators of $p_k(b^{d^m})$ and $q_k(b^{d^m})$ do not increase over $m$.\\
\\
It follows that for each $k$, there exists a constant $N \in \Z$ such that
\begin{align*}
|q_{k, m}(b)|^{-\frac{d_{k+1}(1+\epsilon)}{d_k(1-\epsilon)} + \ln(N)/\ln(|q_{k, m}(b)|)} &< |Nq_{k, m}(b)f(b)-Np_{k, m}(b)| \\
&< |q_{k, m}(b)|^{-\frac{d_{k+1}(1-\epsilon)}{d_k(1+\epsilon)} + \ln(N)/\ln(|q_{k, m}(b)|)},
\end{align*}
where $Nq_{k, m}(b)$ and $Np_{k, m}(b)$ are integers for all $m$. As $m$ increases, so does $|q_{k, m}(b)|$ by Lemma \ref{lemma1.5}, hence when we take the supremum over all exponents, the additional $\ln(N)/\ln(|q_{k, m}(b)|)$ factor tends to zero and will not affect the irrationality exponent. This resolves our problem, and hence shows
\begin{align*}
\mu(f(b)) \geq 1 + \limsup_{k \to \infty} \frac{d_{k+1}}{d_k}.
\end{align*}
We are left to prove the previous claim.\\
\\
\underline{Proof of Claim:}
From (\ref{qkm}), (\ref{pkm}) and our assumptions on the $o(1)$ terms, the claim quickly reduces down to showing the following two inequalities:
\begin{align*}
(d_k + \frac{r_b}{d-1} - \frac{1}{2})\frac{d_{k+1}(1+\epsilon)}{d_k(1-\epsilon)} > d_{k+1}-\frac{r_a}{d-1}+ 1/2\\
d_{k+1}-\frac{r_a}{d-1} - 1/2 > (d_k + \frac{r_b}{d-1} + \frac{1}{2})\frac{d_{k+1}(1-\epsilon)}{d_k(1+\epsilon)}.
\end{align*}
We show both in order. First,
\begin{align*}
&(d_k + \frac{r_b}{d-1} - \frac{1}{2})\frac{d_{k+1}(1+\epsilon)}{d_k(1-\epsilon)} - ( d_{k+1}-\frac{r_a}{d-1}+ 1/2) \\
&= \frac{4\epsilon d_{k+1}d_k - d_{k+1}(1+\epsilon)}{2d_k(1-\epsilon)} + \frac{r_a + r_b}{d-1} - \frac{1}{2}\\
&> \frac{d_{k+1}(\frac{4r_a}{d-1} + 4 - 1 - \epsilon)}{2d_k(1-\epsilon)} + \frac{r_a + r_b}{d-1} - \frac{1}{2}\\
&> \frac{3r_a}{d-1} + \frac{r_b}{d-1} + 1 - \frac{\epsilon}{2}> 0.
\end{align*}
Simply rearranging terms gives the first inequality. For the second inequality,

\begin{align*}
&(d_{k+1}-\frac{r_a}{d-1} - 1/2) - (d_k + \frac{r_b}{d-1} + \frac{1}{2})\frac{d_{k+1}(1-\epsilon)}{d_k(1+\epsilon)}\\
&> d_{k+1}\frac{2\epsilon}{1+\epsilon} - \frac{r_a}{d-1} - \frac{1}{2} - \frac{d_{k+1}(1-\epsilon)}{d_k(1+\epsilon)}(\epsilon d_k-\frac{1}{2})\\
&= d_{k+1}\epsilon - \frac{r_a}{d-1}-\frac{1}{2} + \frac{d_{k+1}(1-\epsilon)}{2d_k(1+\epsilon)}\\
&> \frac{r_a}{d-1}+1 - \frac{r_a}{d-1}-\frac{1}{2}> 0.
\end{align*}
This proves the claim.\\
\\

\end{proof}

\section{Upper bound for $\mu(f(b))$}

Before we evaluate the upper bound, we first require a lemma about irrationality exponents in general. In the previous section, we found a series of rational approximations to $f(b)$ which we were able to use to get a lower bound on the irrationality exponent, however, as far as we know at this point, there may exist better rational approximations which give a larger irrationality exponent. Although it may not seem intuitive, it is possible to show that if there are too many "good" rational approximations to a number, then there cannot exist many "great" rational approximations. This concept will allow us to bound the irrationality exponent from above, which happens to intersect our lower bound found previously. We formalise this idea in the following Lemma.


\begin{Lemma}
Fix $\alpha \in \R$. Assume there exists two sequences, $\big(\frac{p_n}{q_n}\big)_{n \geq 0} \in \Q$ and  $\big(\frac{p_n'}{q_n'}\big)_{n \geq 0} \in \Q$ forming rational approximations for $\alpha$. Furthermore, suppose there exists three real valued sequences $(\theta_n)_{n \geq 0}, (\delta_n)_{n \geq 0}$ and $(\tau_n)_{n \geq 0}$ with $\theta_n$, $\delta_n, \tau_n > 0$ such that for all $n \in \N$,
\begin{enumerate}
    \item $q_n' \ll q_n^{\theta_n}$.\\
    \item $\Big| \alpha - \frac{p_n'}{q_n'}\Big| = \Theta((q_n')^{-1-\tau_n})$ and $\tau_n$ is bounded from above as $n \to \infty$.\footnote{If $\tau_n$ is unbounded, then this condition immediately implies $\mu(\alpha) = \infty$.}\\
    \item $(q_n')^{\tau_n} \gg q_{n+1}^{\delta_{n+1}} \text{ and } q_n^{\delta_n} \to \infty \text{ as } n \to \infty$.
\end{enumerate}
Under these conditions, we have
\begin{align*}
\mu(\alpha) \leq \limsup_{n \to \infty} \max\Big(1 + \frac{\theta_n}{\delta_n}, \frac{(1+\tau_n)\theta_n}{\delta_n}\Big).\\
\end{align*}
\end{Lemma}
Although many of the claims of this report are the work of Badziahin extended to the case of $C(z) \neq 0$, it should be noted that this Lemma is slightly stronger than that of Badziahin's. The differences include the fact that we do not assume anything about how well $\big(\frac{p_n}{q_n}\big)_{n \geq 0}$ approximates $\alpha$ or if $\theta_n \geq 1$ for all $n \in \N$ \cite{badziahin}. Although these extra assumptions will be met for our purposes, it is certainly possible to use this Lemma in contexts other than for Mahler numbers where we don't have them.

\begin{proof}
First assume that $1 + \frac{\theta_n}{\delta_n}$ and $(1+\tau_n)\frac{\theta_n}{\delta_n}$ are bounded from above, as if this is not the case, the statement of the Lemma is obvious as $\mu(\alpha) \leq \infty$ always.\\
\\
Let $\frac{p}{q}$ be any rational number approximating $\alpha$, and assume $q$ is large. Let $c_1$ be the constant from condition (b) such that
\begin{align*}
\Big| \alpha - \frac{p_n'}{q_n'}\Big| \leq \frac{c_1}{(q_n')^{1+\tau_n}},
\end{align*}
for all $n \geq 0$. Fix $n \in \N$ to be the smallest integer such that $2c_1q \leq (q_n')^{\tau_n}$, which is possible as (c) gives us $(q_n')^{\tau_n}$ diverging to $\infty$ as $n \to \infty$. With this, we can bound $q_n'$ from above in terms of $q$. Letting the relevant asymptotic constants being denoted by $c_i$ for $i \geq 1$, 
\begin{align*}
q_n' &\leq c_2q_n^{\theta_n}\\
&= c_2 (q_n^{\delta_n})^{\frac{\theta_n}{\delta_n}}\\
&\leq c_2 (c_3(q_{n-1}')^{\tau_{n-1}})^{\frac{\theta_n}{\delta_n}}\\
&< c_2(2c_1c_3q)^{\frac{\theta_n}{\delta_n}}.
\end{align*}
Consider the following two cases for our approximation $\frac{p}{q}$:\\
\\
\underline{Case 1:} $\frac{p}{q} = \frac{p_n'}{q_n'}$. Then by condition (b) and our bound on $q_n'$, there exists a constant $c_4$ such that
\begin{align*}
\Big| \alpha - \frac{p}{q}\Big| &\geq \frac{c_4}{(q_n')^{1+\tau_n}}\\
&> \frac{c_4}{c_2^{1+\tau_n}(2c_1c_3q)^{(1+\tau_n)\frac{\theta_n}{\delta_n}}}.
\end{align*}
Note that as $\tau_n$ and $(1 + \tau_n)\frac{\theta_n}{\delta_n}$ are bounded from above, we may bound $c_2^{1+\tau_n}(2c_1c_4)^{(1 + \tau_n)\frac{\theta_n}{\delta_n}}$ from above by constants independent of $q$ and $n$. This means there exists some constant $C_1 > 0$ such that

\begin{align*}
\Big| \alpha - \frac{p}{q}\Big| > \frac{C_1}{q^{(1 + \tau_n)\frac{\theta_n}{\delta_n}}}.
\end{align*}
\underline{Case 2:} $\frac{p}{q} \neq \frac{p_n'}{q_n'}$. By the reverse triangle inequality, our chosen $n$ and our bound on $q_n'$,
\begin{align*}
\Big|\alpha - \frac{p}{q}\Big| &\geq \Big|\frac{p}{q}-\frac{p_n'}{q_n'}\Big| - \Big| \alpha - \frac{p_n'}{q_n'}\Big|\\
&\geq \frac{1}{qq_n'} - \frac{c_1}{(q_n')^{1+\tau_n}}\\
&\geq \frac{1}{qq_n'} - \frac{1}{2qq_n'}\\
&> \frac{1}{2qc_2(2c_1c_3q)^{\frac{\theta_n}{\delta_n}}}\\
&= \frac{c_1c_3}{c_2(2c_1c_3q)^{1 + \frac{\theta_n}{\delta_n}}}.
\end{align*}
Again using the fact that $1 + \frac{\theta_n}{\delta_n}$ is bounded from above, there exists some constant $C_2 > 0$ such that

\begin{align*}
\Big|\alpha - \frac{p}{q}\Big| &> \frac{C_2}{q^{1 + \frac{\theta_n}{\delta_n}}}.
\end{align*}
To finish the proof, fix $\mu$ to be some number larger than our claimed upper bound for the irrationality exponent. It follows there exists some $n_0 \in \N$ such that $\mu > 1 + \frac{\theta_n}{\delta_n}$ and $\mu > (1 + \tau_n)\frac{\theta_n}{\delta_n}$ for all $n \geq n_0$. Fix some $n_1 \geq n_0$ such that $(q_n')^{\tau_n} < (q_{n_1}')^{\tau_{n_1}}$ for all $n \leq n_0$. It follows that in any case, if $\frac{p}{q}$ is a rational approximation for $\alpha$ where $2c_1q > (q'_{n_1})^{\tau_{n_1}}$, our fixed choice for $n$ will be larger than $n_0$ hence in all cases, setting $C = \min(C_1, C_2)$,
\begin{align*}
\Big| \alpha - \frac{p}{q}\Big| > \frac{C}{q^\mu}.
\end{align*}
An elementary fact about irrationality exponents is that due the requirement of this inequality being satisfied by infinitely many $q$ integers, the constant $C$ may be ignored in favour of the asymptotic growth of $q^\mu$. Our working implies there is an upper bound on $q$ if we are to attain an irrationality exponent of $\mu$, hence there cannot be an infinite number of sufficiently accurate approximations to $\alpha$. This means the irrationality exponent must be less than $\mu$. Taking the infimum over all choices of $\mu$ gives the required result.
\end{proof}

\begin{Theorem}
\label{thm1.10}
Let $f(z) \in \Q((z^{-1}))$ be a Mahler function that satisfies (\ref{Mahler}) and is not a rational function. Let $b \in \Z, |b| \geq 2$ be inside the disc of convergence of $f$ such that $A(b^{d^t})B(b^{d^t}) \neq 0$ for all $t \in \N$. Then the irrationality exponent of $f(b)$ is given by
\begin{align*}
\mu(f(b)) = 1 + \limsup_{k \to \infty} \frac{d_{k+1}}{d_k}.
\end{align*}
\end{Theorem}

\begin{proof}[Proof of Theorem]
The proof is an application of the previous lemma. Fix some $k_0 \in \N$, and let $K$ be the smallest integer such that $d_{K+1} > d d_{k_0+1}-r_a+d+1$. Recall from the proof of the lower bound, we found that for any $\epsilon > 0$, there exists a $k = k(\epsilon)$ and $m_0 = m(k)$ such that for $m \geq m_0$,

\begin{align*}
| q_{k, m}(b)|^{-\frac{d_{k+1}(1+\epsilon)}{d_k(1-\epsilon)}} < |q_{k, m}(b)f(b)-p_{k, m}(b)| <| q_{k, m}(b)|^{-\frac{d_{k+1}(1-\epsilon)}{d_k(1+\epsilon)}}.
\end{align*}
Now consider a $m > M = M(k_0) = \max(m(k_0), m(k_0+1), ..., m(K))$. It follows there exists a sequence $(\epsilon_{k, m})_{k_0 \leq k \leq K, m \geq M}$ dependent on $k_0$ such that,

\begin{align}
\label{1.11}
\Big| f(b) - \frac{p_{k, m}(b)}{q_{k, m}(b)}\Big| = \Theta\Big(|q_{k, m}(b)|^{-1-(\frac{d_{k+1}}{d_k} + \epsilon_{k, m})}\Big).
\end{align}
Furthermore, $\sup_{k_0 \leq k \leq K, m \geq M}|\epsilon_{k, m}|$ tends to zero as $k_0$ tends to $\infty$. Now consider the following sequences; for any integer $n \geq 1$, write $n$ uniquely in the form $n = u(K-k_0) + v$ where $u, v \in \N, 1 \leq v \leq K-k_0$. Then, 
\begin{align*}
\frac{P_n}{Q_n} &:= \frac{p_{k_0+v-1, M+u}(b)}{q_{k_0+v-1, M+u}(b)}\\
\frac{P_n'}{Q_n'} &:= \frac{p_{k_0+v, M+u}(b)}{q_{k_0+v, M+u}(b)}\\
\delta_{n} &:= \frac{d_{k_0+v}}{d_{k_0+v-1}} + \epsilon_{k_0+v-1, M+u}\\
\tau_{n} &:= \frac{d_{k_0+v+1}}{d_{k_0+v}} + \epsilon_{k_0+v, M+u}\\
\theta_n &:= \frac{d_{k_0+v} + \frac{r_b}{d-1} + o(1)}{d_{k_0+v-1} + \frac{r_b}{d-1} + o(1)} = \frac{d_{k_0+v}}{d_{k_0+v-1}} + \epsilon^*_{k_0+v-1, M+u}.
\end{align*}
Here the $o(1)$ terms refer to the same $o(1)$ term from (\ref{pkm}), and $\epsilon^*_{k_0+v-1, M+u} \to 0$ as $n \to \infty$. We now check each of the conditions of the previous lemma. For condition (a), using (\ref{pkm}),

\begin{align*}
|Q_n'| &= |q_{k_0+v, M+u}(b)|\\
&= b^{d^{M+u}(d_{k_0+v}+\frac{r_b}{d-1} + o(1))}\\
&= b^{d^{M+u}(d_{k_0+v-1}+\frac{r_b}{d-1} + o(1))\theta_n} = |Q_n|^{\theta_n}.
\end{align*}
Now for condition (b), if $\tau_n$ is bounded, the condition clearly holds and using (\ref{1.11}). If $\tau_n$ is not bounded, then the statement of the Theorem is obvious as Lemma \ref{lowerbound} immediately gives $\mu(f(b)) \geq \infty$ and thus we have equality and we don't need to consider this case. For condition (c), it is easy to observe that $\delta_n$ does not converge to zero, and thus by Lemma \ref{lemma1.5}, $Q_n^{\delta_n} \to \infty$. \\
\\
Writing $n = u(K-k_0) + v$ as before, if $v \neq K-k_0$, condition (c) is clear as $(Q_n')^{\tau_n} = Q_{n+1}^{\delta_{n+1}}$. If $n = u(K-k_0)$ for some integer $u$, then by Lemma \ref{lemma1.5} and our definition of $K = K(k_0)$,
\begin{align*}
|q_{K, M+u}(b)f(b) - p_{K, M+u}(b)| &= b^{-d^{M+u}(d_{K+1} - \frac{r_a}{d-1} + o(1))}\\
&< b^{-d^{M+u+1}(d_{k_0+1}-\frac{r_a}{d-1}+o(1))}\\
&= |q_{k_0, M+u+1}(b)f(b) - p_{k_0, M+u+1}(b)|.
\end{align*}
By applying (\ref{1.11}) and the definitions of $\tau_n$ and $\delta_n$, we must have $(Q_n')^{\tau_n} \gg Q_{n+1}^{\delta_{n+1}}$, verifying condition (c) in all cases.\\
\\
Now applying the lemma,
\begin{align*}
\mu(f(b)) \leq \limsup_{u \to \infty} \max_{k_0 \leq v < K}\Big(1 + \frac{\theta_{u(K-k_0) + v}}{\delta_{u(K-k_0)+v}}, (1+ \tau_{u(K-k_0) + v})\frac{\theta_{u(K-k_0)+v}}{\delta_{u(K-k_0)+v}}\Big).
\end{align*}
By definition, $\theta_n/\delta_n$ tends to 1 as $k_0 \to \infty$, however $\tau_{u(K-k_0)+v}$ tends to $d_{k_0+v+1}/d_{k_0+v}$ as $k_0$ tends to $\infty$ since the $\epsilon$ terms approach zero. Since $K-k_0$ increases as $k_0$ tends to $\infty$, it follows

\begin{align*}
\mu(f(b)) \leq \limsup_{n \to \infty}\Big(1 + \frac{d_{n+1}}{d_n}\Big).
\end{align*}
This exactly coincides with the lower bound of Lemma \ref{lowerbound}, hence we attain equality.
\end{proof}


\chapter{Computing the Irrationality Exponent}
While the previous chapter successfully finds an expression for the irrationality exponent in terms of the degrees of the denominator of the convergents of $f$, this expression is not easy to compute. In general, there is no closed formula for what the convergents of a formal Laurent series may be and usually they can only be computed sequentially. This makes any limit involving the convergents of a function very difficult to compute.\\
\\
The aim of this chapter is to investigate the convergents of a Mahler function satisfying (\ref{Mahler}) more deeply, with eventual goal of finding a practical method to calculate what this limit is. Again, much of this chapter achieves the same goals as Badziahin, however we must modify the various statements to account for the fact that $C(z) \neq 0$.\\
\\
Under the conditions of Theorem \ref{thm1.10}, we have found the irrationality exponent of a Mahler number $f(b)$ to be given by
\begin{align*}
\mu(f(b)) = 1 + \limsup_{k \to \infty} \frac{d_{k+1}}{d_k},
\end{align*}
where $d_k = \deg(q_k(z))$, the degree of the denominator of the $k$-th convergent of $f$. As we are essentially studying the differences between the $d_k$, we denote the set of the degrees, $\Phi$.
\begin{align*}
\Phi = \Phi(f) := \{d_k : k \in \N\} \subseteq \N.
\end{align*}
We say the tuple $[u, v]$ is a \textit{gap in $\Phi$ of size $r$} if $u, v \in \Phi$, $r = v-u > 0$ and $(u, v) \cap \Phi = \emptyset$ (that is, $u$ and $v$ are consecutive elements in $\Phi$). For further simplicity, we will also relabel the subscripts of the convergents from now so the subscript is equal to the degree of the denominator, as this is sufficient to identify which convergent we are discussing. That is, $p_u(z)/q_u(z)$ is the convergent of $f$ such that $\deg(q_u(z)) = u$.\\
\\
Note that if the size of all gaps in $\Phi$ are bounded, then $\mu(f(b)) = 2$. This is because for any upper bound on the gap size $C > 0$,
\begin{align*}
\limsup_{k \to \infty} \frac{d_{k+1}}{d_k}
&\leq \lim_{u \to \infty} \frac{u+C}{u} = 1.
\end{align*}
It follows that we can simply consider gaps that are larger than any arbitrary constant. We say a gap $[u, v]$ is \textit{big} if

\begin{align*}
v - u > \frac{r_a + r_b}{d-1}.
\end{align*}
The reason we use this constant specifically is clear in the context of the following lemma.

\begin{Lemma}
\label{lemma2.1}
Let $[u, v]$ be a big gap in $\Phi$. Then (after simplifying) the fraction
\begin{align*}
\frac{A(z)p_u(z^d) + C(z)q_u(z^d)}{B(z)q_u(z^d)},
\end{align*}
is a convergent of $f$. Furthermore, the size of the gap corresponding to this convergent is larger than $v-u$, the size of the gap $[u, v]$.
\end{Lemma}

\begin{proof}
Let $G(z) = \gcd(A(z)p_u(z^d) + C(z)q_u(z^d), B(z)q_u(z^d))$, and let $r_g = \deg(G(z))$. It follows that, after simplifying, the claimed convergent of $f$ is
\begin{align*}
\frac{(A(z)p_u(z^d) + C(z)q_u(z^d))/G(z)}{B(z)q_u(z^d)/G(z)}.
\end{align*}
To show that this is a convergent of $f$, we will use Theorem \ref{criterion}. That is, it's enough to show
\begin{align*}
\deg\Big(\frac{B(z)q_u(z^d)}{G(z)}f(z) - \frac{A(z)p_u(z^d) + C(z)q_u(z^d)}{G(z)}\Big) &< -\deg\Big(\frac{B(z)q_u(z^d)}{G(z)}\Big)\\
&= -r_b - du + r_g.
\end{align*}
We can recall from (\ref{Mahlerapply}),

\begin{align*}
\deg\Big(\frac{B(z)q_u(z^d)}{G(z)}f(z) - \frac{A(z)p_u(z^d) + C(z)q_u(z^d)}{G(z)}\Big) = r_a - dv - r_g,
\end{align*}
thus it is enough to show that $r_a - dv - r_g < -r_b-du+r_g$. Checking this is true, we use the fact that $[u, v]$ is a big gap to find
\begin{align*}
(r_a - dv - r_g) - (-r_b-du+r_g) &= r_a + r_b - d(v-u) - 2r_g\\
&< -(v-u) - 2r_g < 0.
\end{align*}
It follows that the given fraction is a convergent. Applying Lemma \ref{lemma1} and using the fact that $[u, v]$ is a big gap, we find the associated gap size of the new convergent must be
\begin{align*}
(-r_a+dv+r_g) - (r_b+du-r_g) &> (v-u) + 2r_g \geq v-u.
\end{align*}
As the gap size of the new convergent is larger than $v-u$, we are done.\\
\\
\end{proof}
The power of this Lemma is that, given one convergent of $f$ that corresponds to a big gap, we can use it in addition to $A(z), B(z), C(z)$ and $d$ to find another convergent of $f$, and this will correspond to another big gap. Given one big gap, it is possible to apply this lemma repeatedly to produce a sequence of gaps strictly increasing in size, which is intuitively helpful in computing the limit superior from Theorem \ref{thm1.10}.\\
\\
To formalise this, we will introduce some terminology for these gaps.

\begin{Definition}
Let $[u, v]$ be a gap in $\Phi$. We say that the gap $[u', v']$ is the direct successor of the gap $[u, v]$ if
\begin{align*}
\frac{p_{u'}(z)}{q_{u'}(z)} &= \frac{A(z)p_u(z^d) + C(z)q_u(z^d)}{B(z)q_u(z^d)}\\
\end{align*}

We also say $[u, v]$ is a primitive gap if $[u, v]$ is a big gap and there does not exist any other big gap $[u', v']$ such that $[u, v]$ is the direct successor of $[u', v']$.
\end{Definition}

Suppose that $[u_0, v_0]$ is a primitive gap in $\Phi$. By Lemma \ref{lemma2.1}, it is possible to produce an infinite sequence of big gaps $([u_n, v_n])_{n \geq 0}$ such that $[u_{n+1}, v_{n+1}]$ is the direct successor of $[u_n, v_n]$ for all $n \geq 0$. This sequence is hereby called the \textit{primitive gap sequence} from $[u_0, v_0]$.\\
\\
Any primitive gap sequence cannot contain the same gap twice as the gap size is strictly increasing by Lemma \ref{lemma2.1}. Furthermore, $u_n$ and $v_n$ are strictly increasing. To see this, we recall from the proof of Lemma \ref{lemma2.1} that for any big gap $[u_n, v_n]$, we have
\begin{align*}
u_{n+1} = du_n + r_b -r_{g, n}\\
v_{n+1} = dv_n - r_a + r_{g, n},
\end{align*}
where $r_{g, n}$ is the degree of some polynomial, and hence not negative. Recalling $[u_n, v_n]$ is a big gap, 
\begin{align*}
(d-1)(v_n-u_n) &> r_a + r_b\\
\implies dv_n - r_a &> r_b+(d-1)u_n \geq u_n
\end{align*}
It follows $v_{n+1} \geq dv_n-r_a > u_n$. As gaps in $\Phi$ cannot overlap, the $[u_{n+1}, v_{n+1}]$ must occur after, or be the same gap, as $[u_n, v_n]$, however we know it cannot be the same gap as by Lemma \ref{lemma2.1} the gap size is strictly increasing. This only leaves the possibility of the direct successor occurring after.\\
\\
By definition, all big gaps in $\Phi$ are either themselves primitive gaps, or are the direct successor of another big gap which occurs before it. It follows that every big gap falls into at least one primitive gap sequence, hence the set of all primitive gap sequences covers the entire set of big gaps in $\Phi$, From this evidence, you may be inclined to form the hypothesis that the formula deduced from Theorem \ref{thm1.10} can be written as

\begin{align}
\label{hypothesis}
\mu(f(b)) = 1 + \sup_{[u_0, v_0] \text{ is a primitive}}\Big\{\limsup_{n \to \infty} \frac{v_n}{u_n}\Big\} \cup \{1\}.
\end{align}
The union with $\{1\}$ is required if there are no primitive gaps in $\Phi$. If this is true, it would make calculating the limit superior much easier as from the working of Lemma \ref{lemma2.1}, we can write $u_{n+1}$ and $v_{n+1}$ in terms of $u_n$ and $v_n$. This equality is in fact true, however we cannot declare this just yet, as the limit superior is the maximum limit of all possible subsequences, not just the primitive gap sequences. If there are an infinite number of primitive gap sequences, the maximal subsequence may possibly have only finitely many terms across each primitive gap sequence, causing the primitive gap sequences to fail to capture the true maximal subsequence. To show that (\ref{hypothesis}) is correct, we require an additional Lemma.

\begin{Lemma}
\label{lemma2.4}
Let $[u, v]$ be a big gap in $\Phi$ and let $([u_n, v_n])_{n \geq 0}$ be the primitive gap sequence that contains $[u, v]$. If $u\neq 0$ or $r_b \neq 0$, then,
\begin{align*}
\frac{v-\frac{r_a}{d-1}}{u+\frac{r_b}{d-1}} \leq \limsup_{n \to \infty} \frac{v_n}{u_n}. 
\end{align*}
Furthermore, if $u > \frac{r_a}{d-1}$, then

\begin{align*}
\limsup_{n \to \infty} \frac{v_n}{u_n} \leq \frac{v+\frac{r_b}{d-1}}{u - \frac{r_a}{d-1}}.
\end{align*}
\end{Lemma}

\begin{proof}
Recall from the proof of Lemma \ref{lemma2.1}, If $[u_n, v_n]$ is a big gap in $\Phi$, then $[u_{n+1}, v_{n+1}]$ is given by
\begin{align}
\label{lem2.4u}
u_{n+1} &= du_n+r_b-r_{g, n}\\
\label{lem2.4v}
v_{n+1} &= dv_n-r_a+r_{g, n},
\end{align}
where $r_{g, n}$ is the degree of $G_n(z) = \gcd(A(z)p_{u_n}(z^d) + C(z)q_{u_n}(z^d), B(z)q_{u_n}(z^d))$. Note by the Euclidean algorithm and Corollary \ref{corollay},
\begin{align*}
0 \leq r_{g, n} &= \deg(\gcd(A(z)p_{u_n}(z^d) + C(z)q_{u_n}(z^d), B(z)q_{u_n}(z^d)))\\
&\leq \deg(\gcd(A(z)p_{u_n}(z^d) + C(z)q_{u_n}(z^d), q_{u_n}(z^d))) + r_b\\
&\leq \deg(\gcd(A(z), q_{u_n}(z^d))) + r_b \leq r_a + r_b.
\end{align*}
This gives the inequalities,
\begin{align*}
du_n-r_a \leq u_{n+1} \leq du_n+r_b\\
dv_n-r_a \leq v_{n+1} \leq dv_n+r_b.
\end{align*}
Let $m \in \N$ be such that $[u, v] = [u_m, v_m]$, and let $n \in \N$ be larger than $m$. By using the above inequalities a total of $m-n$ times and summing over a finite geometric series, we find,
\begin{align*}
d^{n-m} u_m - \frac{d^{n-m}-1}{d-1} r_a \leq u_n \leq d^{n-m} u_m + \frac{d^{n-m}-1}{d-1}r_b\\
d^{n-m} v_m - \frac{d^{n-m}-1}{d-1} r_a \leq v_n \leq d^{n-m} v_m + \frac{d^{n-m}-1}{d-1}r_b.
\end{align*}
As $u + \frac{r_b}{d-1} > 0$ by assumption, our upper bound on $u_n$ is positive for large enough $n$, thus
\begin{align*}
\frac{d^{n-m}v_m - \frac{d^{n-m}-1}{d-1}r_a}{d^{n-m}u_m + \frac{d^{n-m} - 1}{d-1}r_b} \leq \frac{v_n}{u_n}.
\end{align*}
Taking $n \to \infty$, we find
\begin{align*}
\frac{v - \frac{r_a}{d-1}}{u + \frac{r_b}{d-1}} \leq \limsup_{n \to \infty} \frac{v_n}{u_n}.
\end{align*}
If $u > \frac{r_a}{d-1}$, we then have $d^{n-m} u - \frac{d^{n-m}-1}{d-1} r_a > \frac{r_a}{d-1} \geq 0$ for all large $n \in \N$, hence the lower bound on $u_n$ is positive. Due to this, we can apply a similar manipulation and take $n \to \infty$ to find

\begin{align*}
\limsup_{n \to \infty} \frac{v_n}{u_n} \leq \frac{v+\frac{r_b}{d-1}}{u - \frac{r_a}{d-1}}.
\end{align*}
\end{proof}

With this Lemma, we are now ready to show that the primitive gap sequences of $\Phi$ are enough to evaluate the irrationality exponent formula.

\begin{Lemma}
\label{lemma2.5}
\begin{align}
\label{thm2.5}
\sup_{[u_0, v_0] \text{ is a primitive gap}}\Big\{ \limsup_{n \to \infty} \frac{v_n}{u_n}\Big\} \cup \{1\} = \limsup_{k \to \infty} \frac{d_{k+1}}{d_k}.
\end{align}
\end{Lemma}

\begin{proof}

Clearly all primitive gap sequences form a subset of all possible subsequences of the gaps in $\Phi$, so we must have

\begin{align*}
\sup_{[u_0, v_0] \text{ is a primitive gap}}\Big\{ \limsup_{n \to \infty} \frac{v_n}{u_n}\Big\} \cup \{1\} \leq \limsup_{k \to \infty} \frac{d_{k+1}}{d_k}.
\end{align*}
We will prove the result by contradiction. Suppose that we have a strict inequality. That is,
\begin{align*}
\sup_{[u_0, v_0] \text{ is a primitive gap}}\Big\{ \limsup_{n \to \infty} \frac{v_n}{u_n}\Big\} \cup \{1\} < \limsup_{k \to \infty} \frac{d_{k+1}}{d_k}.
\end{align*}
This implies $\limsup_{k \to \infty} d_{k+1}/d_k > 1$, hence there must be at least one big gap in $\Phi$ otherwise, as justified at the start of the chapter, we would see $\limsup_{k \to \infty} d_{k+1}/d_k = 1$. It follows that there is at least one primitive gap and the union with $\{1\}$ can be ignored.\\
\\
Let $(d_{k_n})_{n \geq 0}$ be a subsequence of degrees such that
\begin{align*}
\lim_{n \to \infty} \frac{d_{k_n+1}}{d_{k_n}} = \limsup_{k \to \infty} \frac{d_{k+1}}{d_k}.
\end{align*}
As the ratio tends to a number greater than one, eventually $[d_{k_n}, d_{k_{n}+1}]$ corresponds to big gaps only. Furthermore, as $d_{k_n}$ tends to $\infty$, we of course have
\begin{align*}
\lim_{n \to \infty} \frac{d_{k_n+1}}{d_{k_n}} = \lim_{n \to \infty} \frac{d_{k_n+1} - \frac{r_a}{d-1}}{d_{k_n} + \frac{r_b}{d-1}}.
\end{align*}
Due to the assumption of a strict inequality, fix $n$ such that  $[d_{k_n}, d_{k_{n+1}}]$ is a big gap with $d_{k_n} > \frac{r_a}{d-1}$ and 
\begin{align*}
\frac{d_{k_n+1} - \frac{r_a}{d-1}}{d_{k_n} + \frac{r_b}{d-1}} > \sup_{[u_0, v_0] \text{ is a primitive gap}}\Big\{ \limsup_{n \to \infty} \frac{v_n}{u_n}\Big\}.
\end{align*}
Let $([u_k, v_k])_{k \geq 0}$ be the primitive gap sequence that contains $[d_{k_n}, d_{k_{n}+1}]$. Then by applying Lemma \ref{lemma2.4}, we find

\begin{align*}
\frac{d_{k_n+1} - \frac{r_a}{d-1}}{d_{k_n} + \frac{r_b}{d-1}} &\leq \limsup_{k \to \infty} \frac{v_k}{u_k} \leq \sup_{[u_0, v_0] \text{ is a primitive gap}}\Big\{ \limsup_{n \to \infty} \frac{v_n}{u_n}\Big\}.
\end{align*}
This is an immediate contradiction to how we chose $n$, hence we cannot have a strict inequality. It follows

\begin{align*}
\sup_{[u_0, v_0] \text{ is a primitive gap}}\Big\{ \limsup_{n \to \infty} \frac{v_n}{u_n}\Big\} \cup \{1\} = \lim_{n \to \infty} \frac{d_{k_n+1}}{d_{k_n}}.
\end{align*}
\end{proof}

While we now know that the irrationality exponent formula from Theorem \ref{thm1.10} can be written as
\begin{align*}
\mu(f(b)) = 1 + \sup_{[u_0, v_0] \text{ is a primitive gap}}\Big\{ \limsup_{n \to \infty} \frac{v_n}{u_n}\Big\} \cup \{1\},
\end{align*}
we are yet to show that this can be calculated in practice. For example, there may be an infinite number of primitive gaps in $\Phi$, the positioning of which are not well known, making the supremum over all such gaps difficult to calculate. This can be overcome with the use of the following Lemma.

\begin{Lemma}
\label{lemma2.3}
The size of a primitive gap cannot exceed $\frac{(2d-1)(r_a+r_b)}{d-1}$.
\end{Lemma}

\begin{proof}
We will prove this by contradiction. Suppose there is a primitive gap $[u, v]$ in $\Phi$ such that $v-u > \frac{(2d-1)(r_a+r_b)}{d-1}$. Fix $w$ to be the largest integer such that $dw < v-r_b$, hence $d(w+1) \geq v-r_b$. As $v > \frac{(2d-1)(r_a+r_b)}{d-1} \geq r_b$, it follows $w \geq 0$, so there exists a gap in $\Phi$, $[s, t]$ such that $s \leq w < t$.\\
\\
\underline{Claim:} $[s, t]$ is not a big gap.\\
\\
If $[s, t]$ were a big gap, then by Lemma \ref{lemma2.1}, the rational function
\begin{align*}
\frac{A(z)p_s(z^d) + C(z)q_s(z^d)}{B(z)q_s(z^d)},
\end{align*}
is a convergent of $f$. Furthermore, by the working of Lemma \ref{lemma2.1}, the gap associated with this convergent is $[ds+r_b-r_g, dt-r_a+r_g]$. We can observe that this gap must intersect the gap $[u, v]$ as
\begin{align*}
ds+r_b-r_g \leq dw+r_b < v,
\end{align*}
and,
\begin{align*}
dt-r_a+r_g \geq d(w+1) - r_a \geq v-r_b-r_a > u,
\end{align*}
where the last inequality follows by the initial assumption on $[u, v]$. \\
\\
As $[u, v]$ and $[ds+r_b-r_g, dt-r_a+r_g]$ are both gaps of $\Phi$ that intersect each other, they must be the same gap. We immediately conclude that $[u, v]$ is the direct successor of $[s, t]$ by definition, and hence $[u, v]$ is not a primitive gap, contradicting the assumption of the lemma. This proves our claim that $[s, t]$ is not a big gap. As $w$ is between $s$ and $t$, we also have $0 \leq w-s \leq \frac{r_a+r_b}{d-1}$\\
\\
Now let $p(z) = A(z)p_s(z^d) + C(z)q_s(z^d)$ and $q(z) = B(z)q_s(z^d)$. Using (\ref{Mahlerapply}), the definition of $w$ and the fact $[s, t]$ is not a big gap, we have

\begin{align*}
\deg(q(z)f(z) - p(z)) &= r_a - dt\\
&\leq r_a -d(s+1)\\
&\leq r_a - d(w-\frac{r_a+r_b}{d-1} + 1)\\
&\leq r_a - (v-r_b) + \frac{d(r_a+r_b)}{d-1}\\
&= \frac{(2d-1)(r_a+r_b)}{d-1} - v\\
&< -u.
\end{align*}
By the classical remainder theorem, we can write $q(z) = a(z)q_u(z) + r(z)$ where $a(z), r(z) \in \Q[z]$ with $\deg(r(z)) < \deg(q_u(z)) = u$. Also let $c(z) = p(z) - a(z)p_u(z)$. As $p_u(z)/q_u(z)$ is a convergent of $f$, Lemma \ref{lemma1} gives us
\begin{align*}
\deg(a(z)q_u(z)f(z) - a(z)p_u(z)) &= \deg(a(z)) - v\\
&= \deg(q(z)) - \deg(q_u(z)) - v\\
&= r_b + ds - u - v\\
&\leq r_b+dw - u - v < -u.
\end{align*}
\underline{Claim:} $r(z) = c(z) = 0$\\
\\
Again, we show this claim is true by contradiction. If $r(z) \neq 0$, then $c(z)/r(z)$ is a rational function with $\deg(r(z)) < \deg(q_u(z)) = u$. If $u = 0$, then we must have $r(z) = 0$. Now assume $u > 0$. If $c(z)/r(z)$ is a convergent of $f(z)$, then $\deg(r(z)) \leq \deg(q_{u'}(z))$ where $[u', u]$ is the gap in $\Phi$ preceding $[u, v]$. As successive convergents form better approximations by Lemma \ref{lemma1}, we have $p_{u'}(z)/q_{u'}(z)$ forming an equal or better approximation for $f(z)$ than $c(z)/r(z)$, hence
\begin{align*}
\deg(r(z)f(z)-c(z)) &\geq \deg(q_{u'}(z)f(z) - p_{u'}(z))\\
&= -u.
\end{align*}
In the case that $c(z)/r(z)$ is not a convergent of $f(z)$, then by Theorem \ref{criterion},

\begin{align*}
\deg(r(z)f(z)-c(z)) \geq -\deg(r(z)) > -u.
\end{align*}
It follows that in all cases, $\deg(r(z)f(z)-c(z)) \geq -u$, hence

\begin{align*}
\deg(q(z)f(z)-p(z)) &= \deg\Big((a(z)q_u(z)f(z) - a(z)p_u(z)) + (r(z)f(z) - c(z))\Big)\\
&= \deg(r(z)f(z) - c(z)) \geq -u.
\end{align*}
This is a contradiction to the fact that $\deg(q(z)f(z)-p(z)) < -u$ as we previously verified, thus $r(z) = 0$. If $c(z) \neq 0$, then we of course have $\deg(r(z)f(z)-c(z)) = \deg(c(z)) \geq 0 \geq -u$. Repeating the previous line of working gives the same contradiction, hence $c(z) = 0$, proving the claim.\\
\\
It follows from the claim that $p(z)/q(z) = p_u(z)/q_u(z)$. As in Lemma \ref{lemma2.1}, let $G(z) = \gcd(A(z)p_s(z^d) + C(z)q_s(z^d), B(z)q_s(z^d))$ and $r_g = \deg(G(z))$. Recall from Corollary \ref{corollay}, $\gcd(p_u(z), q_u(z)) = 1$, hence

\begin{align*}
p_u(z) &= \frac{A(z)p_s(z^d) + C(z)q_s(z^d)}{G(z)}\\
q_u(z) &= \frac{B(z)q_s(z^d)}{G(z)}.
\end{align*}
Now again comparing degrees and using (\ref{Mahlerapply}), $u = r_b + ds - r_g$ and $v = dt-r_a+r_g$. Recall from the proof of Lemma \ref{lemma2.4}, we showed $0 \leq r_g \leq r_a+r_b$, so it follows that the size of the primitive gap $[u, v]$ is

\begin{align*}
v-u &= dt-r_a+r_g - r_b - ds + r_g\\
&\leq d(t-s) + r_a + r_b\\
&\leq \frac{d(r_a+r_b)}{d-1} + r_a + r_b\\
&= \frac{(2d-1)(r_a+r_b)}{d-1}.
\end{align*}
This is a direct contradiction to the initial assumption on the size of the primitive gap $[u, v]$, hence we must conclude that there exists no primitive gap $[u, v]$ such that $v-u > \frac{(2d-1)(r_a+r_b)}{d-1}$.

\end{proof}

Although it may not seem obvious, we now have a clear method to compute the irrationality exponent if we know there is at least one big gap in $\Phi$\footnote{Of course the irrationality exponent is 2 if there are no big gaps, however we have not developed the tools to rigorously prove there are no big gaps in $\Phi$ for a given Mahler function.}. If there is at least one big gap in $\Phi$, there must be at least one primitive gap.  Let $[u_0, v_0]$ be the first big gap in $\Phi$ such that $u_0 \neq 0$. By Lemma \ref{lemma2.4}, we can bound how much the primitive gap sequence containing $[u_0, v_0]$ contributes to the irrationality exponent.

\begin{align*}
\limsup_{n \to \infty} \frac{v_n}{u_n} \geq \frac{v_0 - \frac{r_a}{d-1}}{u_0 + \frac{r_b}{d-1}}.
\end{align*}
By Lemma \ref{lemma2.3}, the size of a primitive gap in $\Phi$ is bounded from above by a constant, which we call $S$. Let $[u, v]$ be any primitive gap that occurs after $[u_0, v_0]$. As $[u_0, v_0]$ occurs first and is itself a big gap we must have $u \geq v_0 > \frac{r_a}{d-1}$. Applying Lemma \ref{lemma2.4}, the primitive gap sequence from $[u, v]$ can only contribute to the supremum if

\begin{align*}
\frac{u + S + \frac{r_b}{d-1}}{u - \frac{r_a}{d-1}} > \frac{v_0 - \frac{r_a}{d-1}}{u_0 + \frac{r_b}{d-1}}.
\end{align*}
Note the right hand side of this inequality is strictly larger than one since $v_0 - u_0 > \frac{r_a + r_b}{d-1}$, but the left hand side will tend to one from above as $u$ increases. This implies our inequality is an upper bound for $u$, and hence there are only finitely many primitive gap sequences that could possibly contribute to the irrationality exponent, and the possible locations of the starting primitives are specified by this inequality. To compute the irrationality exponent, all we are left to do is evaluate
\begin{align*}
1 + \limsup_{n \to \infty}\frac{v_n}{u_n}.
\end{align*}
for a finite number of primitive gap sequences $([u_n, v_n])_{n \geq 0}$. Since $u_{n+1}$ and $v_{n+1}$ can be expressed in terms of $u_n$ and $v_n$ for all $n \geq 0$, computing this limit is usually achievable in practice.

\chapter{A Worked Example}
We will now focus on an example application of our work so far. We have outlined the procedure to determine how to calculate the irrationality exponent of many Mahler numbers associated with (\ref{Mahler}), but our procedure relies on the fact that:
\begin{itemize}
    \item We can find one big gap in $\Phi(f)$.
    \item $r_{g, n} = \deg( \gcd(A(z)p_n(z^d) + C(z)q_n(z^d), B(z)q_n(z^d)))$ is known for all $n \geq 0$ and for each primitive gap sequence that contributes to the irrationality exponent.
\end{itemize}
These facts can only be checked in a case by case manner. In this chapter, we will explore the irrationality exponent of Mahler numbers when $A(z), B(z)$ and $C(z)$ are linear polynomials and $d = 3$. Handling the above concerns for all possible Mahler functions of this form still remains very difficult, and in the interests of brevity, it is not worthwhile managing these issues across all examples. Instead, we will only focus on examples where $\Phi(f)$ has big gaps that occur quickly, and the $r_{g, n}$ sequence can be computed with relative ease, handling both of the initial concerns.\\
\\
Suppose $f(z) \in \Q((z^{-1}))$ satisfies
\begin{align*}
f(z) &= \frac{(a_1z + a_0)f(z^3) + (c_1z+c_0)}{z+b_0},
\end{align*}
where $a_1, a_0, b_0, c_1, c_0 \in \Q$, $a_1z+a_0 \neq 0$ and $c_1z + c_0 \neq 0$. Checking what the degree of $f$ may be first,
\begin{align*}
1 + \deg(f(z)) = \deg((a_1z + a_0)f(z^3) + (c_1z+c_0)).
\end{align*}
There cannot be any solution for $f$ if $\deg(f(z)) > 0$ or $\deg(f(z)) < 0$, so we must have $\deg(f(z)) = 0$. Now expanding the function into a formal Laurent series and letting $A_i \in \Q$ denote the coefficients,

\begin{align*}
(z + b_0)\Big(A_0 + \frac{A_1}{z} + \frac{A_2}{z^2} +  ... \Big) = (a_1z+a_0)\Big(A_0 + \frac{A_1}{z^3} + \frac{A_2}{z^6} + ...\Big) + c_1z+c_0.
\end{align*}
By comparing the coefficients of each formal Laurent series, we find $A_0 = a_1A_0 + c_1$, and $b_0A_0 + A_1 = a_0A_0 + c_0$. Additionally, for $n \geq 1$,
\begin{align*}
b_0A_n + A_{n+1} &= \begin{cases}
a_0A_{n/3} &\text{ if $n \equiv 0 \pmod{3}$}\\
0 &\text{ if $n \equiv 1 \pmod{3}$}\\
a_1A_{(n+1)/3} &\text{ if $n \equiv 2 \pmod{3}$.}
\end{cases}
\end{align*}
Note that if $a_1 \neq 1$, $A_0$ is uniquely determined and $A_{n+1}$ can be written in terms of the terms that come before it. It follows that there is a unique solution for $f(z)$ in this case. Furthermore, as $A_{n+1}$ can be expressed as a linear combination of $A_1, A_2, ..., A_n$, we must have $A_n$ being a multiple of $A_1$ for all $n \geq 1$. This means we need $A_1 \neq 0$ for $f$ to be a non-rational function. As rational functions are not interesting, we will only study cases where,
\begin{align*}
A_1 = (a_0-b_0)A_0 + c_0 \neq 0.
\end{align*}
Recall, by definition, the first convergent of $f$ is $p_0(z)/q_0(z) = A_0$. Furthermore, as $A_1 \neq 0$. $\deg(f(z)-A_0) = -1$, hence by Lemma \ref{lemma1}, the next convergent must have a linear denominator, and the first gap in $\Phi(f)$ is $[0, 1]$.\\
\\
As we are interested in when $\Phi$ has large gaps quickly, we will search for conditions for when $[1, 2]$ is not a gap in $\Phi$, meaning the gap following $[0, 1]$ has a size of at least two. Let $q_1(z) = z+d_0$ after rescaling by a constant. Again by Lemma \ref{lemma1}, for this convergent to correspond to a gap of size at least two, we need the existence of a linear polynomial $p_1(z)$ such that $\deg(q_1(z)f(z)-p_1(z)) \leq -3$. That is, 
\begin{align*}
\deg\Big((z+d_0)(A_0 + \frac{A_1}{z} + \frac{A_2}{z^2} + ...) - p_1(z)\Big) \leq -3.
\end{align*}
As $p_1(z)$ is a polynomial, it is fixed to be $\lfloor q_1(z)f(z)\rfloor$ in order for the above expression to have negative degree. We also require the coefficient for $1/z$ and $1/z^2$ to be zero. This leads to the equations,
\begin{align*}
d_0A_1 + A_2 &= 0\\
d_0A_2 + A_3 &= 0.
\end{align*}
Using the recursive formula for the first few coefficients of $f$ and using the fact that $A_n$ is a multiple of $A_1 \neq 0$ for $n \geq 1$, we find

\begin{align*}
d_0  -b_0 &= 0\\
-b_0d_0 + a_1+b_0^2 &= 0.
\end{align*}
This gives the conditions $d_0 = b_0$ and $a_1 = 0$. If we desire a gap larger than this, it is further required that the $1/z^3$ coefficient is zero as well. By similar working, this translates to the condition,
\begin{align*}
(a_1+b_0^2)d_0 + (a_0-b_0^3-a_1b_0) = 0.
\end{align*}
This gives $a_0 = 0$, however this is not desired as it results in $A(z) = 0$, and hence $f(z)$ is a rational function. It follows the next gap in $\Phi$ is $[1, 3]$ in the case $a_1 = 0$ and $a_0 \neq 0$.\\
\\
This is enough working to produce an example of interest.

\begin{Theorem}
\label{application}
Let $f(z) \in \Q((z^{-1}))$ be the unique formal Laurent series that satisfies
\begin{align}
\label{example}
f(z) = \frac{a_0f(z^3) + c_1z + c_0}{z + 1},
\end{align}
where $a_0, c_1, c_0 \in \Q$ such that
\begin{enumerate}
    \item $(a_0-1)c_1 + c_0 \neq 0$.\\
    \item $a_0 \neq 0$.\\
\end{enumerate}
For any $b \in \Z$, $|b| \geq 2$ inside the disc of convergence of $f$, we have $\mu(f(b)) = 3$. Additionally, if conditions (a) or (b) are not met, $f(z)$ is a rational function.

\end{Theorem}

We have fixed the free variable $b_0 = 1$. This was done so the $r_{g, n}$ sequence is easier to manage, as we will observe in the following proof.

\begin{proof}
First, we need to check if $f(z)$ is a rational function, as there is no sense in applying the working of the previous chapters in this case. We use contradiction. Suppose $f(z) = P(z)/Q(z)$ where $P(z)$ and $Q(z)$ are coprime. (\ref{example}) gives us
\begin{align*}
\frac{P(z)}{Q(z)} = \frac{a_0P(z^3) + (c_1z+c_0)Q(z^3)}{(z+1)Q(z^3)}.
\end{align*}
As the numerator and denominator on the left hand side are coprime, there must be some factors to cancel on the right hand side so the expression is the same after simplification. Letting $r_q$ be the degree of $Q(z)$, the degree of the canceled terms must be $3r_q + 1 - r_q = 2r_q + 1$. Note however, the highest degree of cancellation is given by
\begin{align*}
&\deg(\gcd(a_0P(z^3) + (c_1z+c_0)Q(z^3), (z+1)Q(z^3)))\\
&\leq \deg(\gcd(a_0P(z^3) + (c_1z+c_0)Q(z^3), Q(z^3)) (z+1))\\
&\leq \deg(\gcd(aP(z^3), Q(z^3))) + 1\\
&= 1.
\end{align*}
It follows that at most, $r_q = 0$, as otherwise it is impossible to cancel enough terms from the numerator and denominator. This means that $Q(z)$ is a constant, and hence our solution for $f(z)$ must be a polynomial, $P(z)$. As we have already verified $\deg(f(z)) = 0$, we conclude that $f(z) = f$ is a constant. Rearranging (\ref{example}) with this in mind,
\begin{align*}
zf + f = c_1z + a_0f + c_0.
\end{align*}
By comparing coefficients, we have $f = c_1$ and $f(a_0-1)+c_0 = 0$, hence $c_1(a_0-1)+c_0 = 0$. This however contradicts condition (a), hence we cannot have $f(z)$ as a constant, and we finally conclude that $f(z)$ is not a rational function. We can now proceed with finding the irrationality exponent.\\
\\
By our previous working, we already have $A_0 = c_1$ and $A_1 = (a_0-1)c_1 + c_0 \neq 0$ by assumption. Furthermore, the first convergent is $p_0(z)/q_0(z) = c_1$, and the first two gaps of $\Phi$ are $[0, 1]$ and $[1, 3]$.\\
\\
Now, using the procedure outlined at the end of the previous chapter, we start by searching for the first big gap in $\Phi$. The condition to have a big gap is that the size of the gap is at least $\frac{r_a+r_b}{d-1} = \frac{1}{2}$, hence all gaps in $\Phi$ are big, and $[0, 1]$ is the first primitive gap.\\
\\
By Lemma \ref{lemma2.3}, the size of a primitive gap is bounded by $\frac{(2d-1)(r_a+r_b)}{d-1} = \frac{5}{2}$, hence the largest size a primitive gap could be is 2. If $[u, v]$ is any primitive gap that is not $[0, 1]$, then, by Lemma \ref{lemma2.4}, the primitive gap sequence from $[u, v]$ cannot affect the irrationality exponent, except when
\begin{align*}
\frac{u + 2 + \frac{1}{2}}{u - 0} > \frac{1 - 0}{0 + \frac{1}{2}} = 2.
\end{align*}
Solving this inequality for $u \geq 1$, we find that $u < 5/2$, hence $u \leq 2$. It follows that $[1, 3]$ may affect the irrationality exponent if it is a primitive gap, otherwise only the primitive gap sequence from $[0, 1]$ will give the irrationality exponent.\\
\\
By Lemma \ref{lemma2.1}, $[0, 1]$ being a big gap corresponding to the convergent $p_0(z)/q_0(z) = c_1$ implies that
\begin{align*}
\frac{a_0(c_1) + (c_1z+c_0)(1)}{(z+1)(1)} = \frac{c_1z + a_0c_1 + c_0}{z+1}.
\end{align*}
is a convergent of $f(z)$. There is cancellation if $c_1 = a_0c_1+c_0$, however by condition (a) is not true, hence there is no cancellation. As the degree of the denominator is 1, this must correspond to the $[1, 3]$ gap, implying that the $[1, 3]$ gap is not a primitive gap in $\Phi$ by definition and is the direct successor of $[0, 1]$. It follows the only primitive gap sequence that contributes to the irrationality exponent is the sequence from $[0, 1]$.\\
\\
Let $([u_n, v_n])_{n \geq 0}$ be the primitive gap sequence from $[0, 1]$. We claim that $r_{g, n} = 0$ for all $n \geq 0$. That is,
\begin{align*}
\gcd(a_0p_{u_n}(z^3) + (c_1z+c_0)q_{u_n}(z^3), (z+1)q_{u_n}(z^3)) = 1.
\end{align*}
We show this by induction. For $n = 0$, this is relatively straightforward as using (a),
\begin{align*}
\gcd(a_0p_{u_0}(z^3) + (c_1z+c_0)q_{u_0}(z^3), (z+1)q_{u_0}(z^3)) &= \gcd(c_1z + a_0c_1 + c_0, z+1) = 1.
\end{align*}
Assume this holds for all $k \leq n$ for some fixed $n \geq 0$. This means there is no cancellation when computing the direct successor of each convergent, hence for all $k \leq n$,

\begin{align*}
p_{u_{k+1}}(z) &= a_0p_{u_{k}}(z^3) + (c_1z+c_0)q_{u_k}(z^3)\\
q_{u_{k+1}}(z) &= (z+1)q_{u_k}(z^3).
\end{align*}
Applying this and the Euclidean algorithm,

\begin{align*}
&\gcd(a_0p_{u_{n+1}}(z^3) + (c_1z+c_0)q_{u_{n+1}}(z^3), (z+1))\\
&= \gcd(a_0(a_0p_{u_{n}}(z^9) + (c_1z^3+c_0)q_{u_{n}}(z^9)) + (c_1z+c_0)(z^3+1)q_{u_{n}}(z^9), (z+1))\\
&= \gcd(a_0p_{u_{n}}(z^9) + (c_1z^3+c_0)q_{u_{n}}(z^9), (z+1)).
\end{align*}
Now suppose that this $\gcd$ is not constant, so $z+1$ is a factor of the first argument. This implies

\begin{align*}
a_0p_{u_{n}}((-1)^9) + (c_1(-1)^3+c_0)q_{u_{n}}((-1)^9) = 0.
\end{align*}
As $(-1)^9 = (-1)^3 = -1$, we also find that $z + 1$ is a factor of $a_0p_{u_n}(z^3) + (c_1z+c_0)q_{u_n}(z^3)$, which contradicts the inductive hypothesis for $k = n$. This implies that $z+1$ is coprime to $a_0p_{u_{n+1}}(z^3) + (c_1z+c_0)q_{u_{n+1}}(z^3)$. Additionally, by the Euclidean algorithm and Corollary \ref{corollay},

\begin{align*}
&\gcd(a_0p_{u_{n+1}}(z^3) + (c_1z+c_0)q_{u_{n+1}}(z^3), q_{u_{n+1}}(z^3))\\
&= \gcd(a_0p_{u_{n+1}}(z^3), q_{u_{n+1}}(z^3)) = 1.
\end{align*}
Putting this together, we have proved the inductive step and hence completed the induction as

\begin{align*}
\gcd(a_0p_{u_{n+1}}(z^3) + (c_1z+c_0)q_{u_{n+1}}(z^3), (z+1)q_{u_{n+1}}(z^3)) = 1.
\end{align*}
As $r_{g, n} = 0$ for all $n \geq 0$, we have from (\ref{lem2.4u}) and (\ref{lem2.4v}),
\begin{align*}
u_{n+1} &= du_n + r_b = 3u_n + 1\\
v_{n+1} &= dv_n - r_a = 3v_n.
\end{align*}
As $u_0 = 0$ and $v_0 = 1$, we can quickly find closed formulae for $u_n$ and $v_n$
\begin{align*}
u_n &= \frac{3^n-1}{2}\\
v_n &= 3^n.\\
\end{align*}
It follows that taking $n \to \infty$,
\begin{align*}
\limsup_{n \to \infty} \frac{v_n}{u_n} &= 2\lim_{n \to \infty}\frac{3^n}{3^n-1} = 2.
\end{align*}
To properly apply Theorem \ref{thm1.10}, we must check $A(b^{3^t})B(b^{3^t}) \neq 0$ for all integers $t \geq 0$, however as $A$ is a non-zero constant and $B$ only has a root at $z = -1$, this is obvious for any $|b| \geq 2$. It follows that the irrationality exponent of $f(b)$, assuming $f$ converges at $b$, is
\begin{align*}
\mu(f(b)) = 1 + 2 = 3.
\end{align*}

\end{proof}


\chapter{The Set of Irrationality Exponents}
While we have established how to compute the irrationality exponent of many Mahler numbers originating from (\ref{Mahler}), a final remaining question is what are the possible values of irrationality exponents. More specifically, in the context of the original question proposed by Adamczweski and Rivoal, it is possible to show the irrationality exponent of an irrational Mahler number is always rational. From the work of Chapters 2 and 3, we have found that under the conditions of Theorem \ref{thm1.10}, if $\mu(f(b)) > 2$, then

\begin{align*}
\mu(f(b)) = 1 + \sup_{[u_0, v_0] \text{ is a primitive gap}} \Big\{\limsup_{n \to \infty }  \frac{v_n}{u_n}\Big\}.
\end{align*}
Furthermore, only finitely many primitive gap sequences could contribute to this supremum and we have recursive formulae for $u_{n+1}$ and $v_{n+1}$ involving $d, r_a, r_b$ and $r_{g, n}$. The behaviours of each of these terms are well known, with the exception of 

\begin{align*}
r_{g, n} := \gcd(A(z)p_{u_n}(z^d) + C(z)q_{u_n}(z^d), B(z)q_{u_n}(z^d)).
\end{align*}
While we can bound $r_{g, n}$ between 0 and $r_a+r_b$, it's exact value for each $n$ can only be computed on a case by case basis. Badziahin has proven that in the case where $C(z) = 0$, the sequence is eventually periodic \cite{badziahin}. Suppose that this is also true in the case where $C(z) \neq 0$. This means there exists a $n_0 \in \N$ and period $P \in \N$ such that $r_{g, n_0 + k} = r_{g, n_0 + k + P}$ for all $k \geq 0$. Using (\ref{lem2.4u}),
\begin{align*}
u_{n_0 + P} &= du_{n_0 + P-1} + r_b - r_{g, n_0+P-1}\\
&= d^2u_{n_0 + P-2} + r_b(d + 1) - dr_{g, n_0+P-2} - r_{g, n_0+P-1}\\
&\vdots\\
&= d^P u_{n_0} + r_b\frac{d^P-1}{d-1} - \sum_{k=0}^{P-1} d^{P-1-k}r_{g, n_0+k}.
\end{align*}
Similarly,

\begin{align*}
v_{n_0+P} &= d^P v_{n_0} - r_a \frac{d^P-1}{d-1} + \sum_{k=0}^{P-1} d^{P-1-k}r_{g, n_0+k}.
\end{align*}
If we define the terms $r_u$ and $r_v$ as
\begin{align*}
r_u &= r_b\frac{d^P-1}{d-1} - \sum_{k=0}^{P-1} d^{P-1-k}r_{g, n_0+k}\\
r_v &=  r_a \frac{d^P-1}{d-1} - \sum_{k=0}^{P-1} d^{P-1-k}r_{g, n_0+k},
\end{align*}
then by utilising the fact that $r_{g, n}$ is periodic with period $P$,

\begin{align*}
\lim_{k \to \infty} \frac{v_{n_0 + kP}}{u_{n_0 + kP}} &= \lim_{k \to \infty} \frac{d^{kP}v_{n_0} + (1 + d^P + d^{2P} + ... + d^{(k-1)P})r_v}{d^{kP}u_{n_0} - (1 + d^P + d^{2P} + ... + d^{(k-1)P})r_u}\\
&= \frac{v_{n_0} + \frac{r_v}{d-1}}{u_{n_0} - \frac{r_u}{d-1}}.
\end{align*}
This argument can also be applied to the ratio $v_{n_0 + i + kP}/u_{n_0+i+kP}$ for any $0 \leq i \leq P-1$. By applying a Bolzano-Weierstrass style argument, we note one of these subsequences must contain infinitely many terms of any subsequence of $v_n/u_n$, it immediately follows that the limit superior of $v_n/u_n$ must be equal to the limit of one of these sequences, which is clearly a rational number. Furthermore, as only finitely many primitive gaps contribute to the irrationality exponent, it follows that under the assumptions of Theorem \ref{thm1.10}, the irrationality exponent must be rational.\\
\\
This result is exactly the one desired, however it has only been successfully proven in the case that $C(z) = 0$ by Badziahin \cite{badziahin}. The proof that the sequence $r_{g, n}$ is eventually periodic in the case $C(z) = 0$ relies on analysing the different factors of $A(z)p_{u_n}(z^d) + C(z)q_{u_n}(z^d)$, however this immediately fails in the case $C(z) \neq 0$ as factors are not well preserved in addition, with new factors possibly being produced unexpectedly. This requires us to apply several new techniques, from which we can partially answer this question.\\
\\
The goal of this chapter is to find (and prove) some conditions under which we may expect $r_{g, n}$ to be periodic, thereby allowing the irrationality exponent to be rational. Additionally we desire conditions that can be checked so the user of this work does not need to perform a significant amount of additional working to check if the irrationality exponent of a Mahler number is certainly rational.\\
\\
One rather simplistic pathway to showing the irrationality exponent is rational is given by the following Lemma.

\begin{Lemma}
Let $f(z)$ be a Mahler function that satisfies the conditions of Theorem \ref{thm1.10} and fix $b \in \Z, |b| \geq 2$. Suppose there exists functions $\lambda_1(z), \lambda_2(z) \in \Q[z]$ such $\lambda_1(b^{d^t}) \neq 0$ for all $t \geq 0$ and
\begin{align*}
\lambda_1(z)\lambda_1(z^d)C(z) + \lambda_1(z^d)\lambda_2(z)B(z) = \lambda_1(z)\lambda_2(z^d)A(z).
\end{align*}
Then $\mu(f(b))$ is rational if $f(b)$ is irrational.
\end{Lemma}

\begin{proof}
We can immediately show that $g(z) = \lambda_1(z)f(z) + \lambda_2(z)$ is a Mahler function too.

\begin{align*}
\lambda_1(z)f(z) + \lambda_2(z) &= \frac{\lambda_1(z)A(z)f(z^d) + \lambda_1(z)C(z) + \lambda_2(z)B(z)}{B(z)}\\
&= \frac{\lambda_1(z)\lambda_1(z^d)A(z)f(z^d) + \lambda_1(z)\lambda_1(z^d)C(z) + \lambda_2(z)\lambda_1(z^d)B(z)}{\lambda_1(z^d)B(z)}\\
&= \frac{\lambda_1(z)\lambda_1(z^d)A(z)f(z^d) + \lambda_1(z)\lambda_2(z^d)A(z)}{\lambda_1(z^d)B(z)}\\
g(z) &= \frac{\lambda_1(z)A(z)g(z^d)}{\lambda_1(z^d)B(z)}.
\end{align*}
It naturally follows that $g(z)$ is a Mahler function with the corresponding $C(z)$ polynomial equal to zero. As $\lambda_1(b^{d^t}) \neq 0$ for all $t \geq 0$, we have $\lambda_1(b^{d^t})A(b^{d^t})\lambda_1(b^{d^{t+1}})B(b^{d^t}) \neq 0$ for all $t \geq 0$, hence $g(b)$ satisfies the conditions of Theorem \ref{thm1.10}. Noting $\lambda_1(b) \neq 0$ and $\lambda_2(b)$ are rational, we have $g(b) = \lambda_1(b)f(b) + \lambda_2(b)$ being irrational, and hence $\mu(g(b))$ is rational by Badziahin \cite{badziahin}. Using the elementary fact that irrationality exponents are invariant under scalar multiplication by rationals and addition of rationals, we have $\mu(g(b)) = \mu(f(b))$, hence $\mu(f(b))$ is rational.
\end{proof}

While a rather direct approach in applying the work of Badziahin, we cannot expect the functions $\lambda_1(z), \lambda_2(z)$ to exist in all cases\footnote{The example from Theorem \ref{application} is one such case.}, as this Lemma simply implies that it may be possible to transform some Mahler functions to a form where $C(z) \neq 0$ while preserving the irrationality exponent. For this reason, we resort to a second method of attempting to show the $r_{g, n}$ sequence is periodic directly that is much more direct.\\
\\
First observe one can write the second argument in the form
\begin{align*}
B(z)q_{u_n}(z^d) &= B_0(z)B_m(z)B_c(z) q_{u_n, 0}(z^d)q_{u_n, m}(z^d)q_{u_n, c}(z^d),
\end{align*}
where we have factored $B(z)$ and $q_{u_n}(z^d)$ into different parts. $B_0(z)$ corresponds to the powers of $z$ in $B(z)$, $B_m(z)$ corresponds to the non-cyclotomic factors and $B_c(z)$ corresponds to the remaining cyclotomic factors. We also do the same to $q_{u_n}(z^d)$. The cyclotomic polynomials are the monic polynomials that are irreducible over $\Q[z]$ and have roots only on the unit circle, which we discuss in more detail later. Note that these factors are coprime, hence we can write
\begin{align*}
G_n(z) &=  G_{n, 0}(z)G_{n, m}(z)G_{n, c}(z)\\
G_{n, 0}(z) &:= \gcd(A(z)p_{u_n}(z^d) + C(z)q_{u_n}(z^d), B_0(z)q_{u_n, 0}(z^d))\\
G_{n, m}(z) &:= \gcd(A(z)p_{u_n}(z^d) + C(z)q_{u_n}(z^d), B_m(z)q_{u_n, m}(z^d))\\
G_{n, c}(z) &:= \gcd(A(z)p_{u_n}(z^d) + C(z)q_{u_n}(z^d), B_c(z)q_{u_n, c}(z^d)).
\end{align*}
These factors are also independent of each other over $n$. That is, any powers of $z$ in $B_0(z)q_{u_n, 0}(z^d)$ will not influence any of the factors of $B_m(z)q_{u_k, m}(z^d)$ or $B_c(z)q_{u_k, c}(z^d)$ for all $k \in \N$ as if $P(z)$ is a power of $z$, then so is $P(z^d)$. This also holds for the cyclotomic and non-cyclotomic factors, and implies each expression can be managed independently. Furthermore, by taking the degree of $G_n(z)$, we find we have split $r_{g, n}$ into the sum of three sequences $r_{g, n} = r_{g, n, 0} + r_{g, n, m} + r_{g, n, c}$, so it is enough for us to find conditions for which each of these three sequences are eventually periodic, as the sum of a finite number of eventually periodic sequences is eventually periodic.

\section*{Powers of $z$}

If $B_0(z)q_{u_0, 0}(z^d)$ is a constant, then we are done as $q_{u_n, 0}(z) | B_0(z)q_{u_{n-1}, 0}(z^d)$, so each successive term of the sequence $B_0(z)q_{u_n}(z^d)$ will be a constant and $r_{g, n, 0}$ will be the constant zero sequence which is periodic. We will deal with three alternate cases depending on where the powers of $z$ may be distributed.\\
\\
\underline{Case 1:} $B_0(z)$ is constant but $q_{u_0, 0}(z^d)$ is not constant. If for any $k \geq 0$, we find $q_{u_k, 0}(z)$ is constant, then by the fact that $q_{u_{k+1}, 0}(z) | B_0(z)q_{u_k, 0}(z^d)$, we find that $r_{g, n, 0} = 0$ for all $n \geq k$, and hence the sequence is eventually periodic. Now consider the other case, where $q_{u_n, 0}(z^d)$ contains a power of $z$ for all $n \geq 0$. This implies $p_{u_n}(z^d)$ cannot contain a power of $z$ by Corollary \ref{corollay} and hence by the Euclidean algorithm,

\begin{align*}
\gcd(A(z)p_{u_n}(z^d) + C(z)q_{u_n}(z^d), B_0(z)q_{u_n, 0}(z^d)) &= \gcd(A(z)p_{u_n}(z^d), q_{u_n, 0}(z^d))\\
&= \gcd(A(z), q_{u_n, 0}(z^d))
\end{align*}
If $\gcd(A(z), q_{u_n, 0}(z^d)) = q_{u_n, 0}(z^d)$ for some $n$, then we immediately have all powers of $z$ being removed from $q_{u_n}(z^d)$ and $q_{u_{n+1}, 0}(z)$ will contain no powers of $z$, a contradiction to our previous assumption. It follows that for all $n$, $\gcd(A(z), q_{u_n, 0}(z^d)) = a$ where $a$ is such that $z^a | A(z)$ but $z^{a+1}$ does not divide $A(z)$. As $a$ is constant, $r_{g, n, 0} = a$ for all $n \geq 0$ and the sequence is eventually periodic.\\ 
\\
\underline{Case 2:} $B_0(z)$ is not constant, but $A(z)$ contains no powers of $z$. If we assume that for all $n \geq 0$, we have $A(z)p_{u_n}(z^d) + C(z)q_{u_n}(z^d)$ containing at least as many powers of $z$ as $B_0(z)q_{u_n, 0}(z^d)$, the gcd will extract all powers of $z$ from $B_0(z)q_{u_n, 0}(z^d)$ for all $n \geq 0$ and $q_{u_n}(z)$ will contain no powers of $z$ for all $n \geq 1$. This means $G_{n, 0}(z) = B_0(z)$ for all $n \geq 1$, hence $r_{g, n, 0} = \deg(B_0(z))$ for all $n \geq 1$, and the sequence is therefore eventually periodic.\\
\\
Now consider the situation where for some $k \in \N$ we have $A(z)p_{u_k}(z^d) + C(z)q_{u_k}(z^d)$ containing fewer powers of $z$ than $B_0(z)q_{u_k, 0}(z^d)$. This implies $p_{u_{k+1}}(z)$ will not contain any powers of $z$, but $q_{u_{k+1}}(z)$ will. As $A(z)$ does not contain any powers of $z$, it follows $A(z)p_{u_{k+1}}(z^d) + C(z)q_{u_{k+1}}(z^d)$ does not contain any powers of $z$, however $B_0(z)q_{u_{k+1}, 0}(z^d)$ will, hence we return to the same situation but for $k+1$. It follows that this situation will repeat, and as $A(z)p_{u_{n}}(z^d) + C(z)q_{u_n}(z^d)$ will not contain any powers of $z$ for $n \geq k+1$, $G_{n, 0}(z) = 1$ for $n \geq k+1$, hence $r_{g, n, 0}$ will eventually be the constant zero sequence in this case.\\
\\
\underline{Case 3:} $B_0(z)$ is not constant and $A(z)$ contains at least one power of $z$. While this is the most challenging case for the powers of $z$ if approached directly, it is possible to entirely avoid this case through the manipulation of the original Mahler functional equation (\ref{Mahler}). Given $f(z)$ satisfies (\ref{Mahler}), we also have $g(z) := f(z)z^K$  being a Mahler function for any $K \in \Z$ as

\begin{align}
f(z) = \frac{A(z)f(z^d) + C(z)}{B(z)} \implies g(z) = \frac{A(z)g(z^d) + C(z)z^{dK}}{B(z)z^{(d-1)K}} \label{transform}
\end{align}
If $A(z)$ contains some powers of $z$, then all three of $A(z)$, $B'(z) := B(z)z^{(d-1)K}$ and $C'(z) := C(z)z^{dK}$ contain powers of $z$, hence we can naturally simplify the above functional equation. If $K$ is large enough, the powers of $z$ in $A(z)$ are cancelled entirely, resulting in $g(z)$ being a Mahler function where the associated $A$ polynomial has no powers of $z$. In this situation we can resort to the previously outlined case 2.\\
\\
If we successfully show that the entire $r_{g, n}$ sequence for $g(z)$ is periodic, we will successfully show $g(b) = f(b)b^K$ has a rational irrationality exponent. Note however, that $b^K$ is an integer and irrationality exponents are unaffected by scaling of a non-zero rational number, hence $f(b)$ has the same irrationality exponent as $g(b)$, which is rational. It follows that it's enough to consider the Mahler functions $f(z)$ for which $A(z)$ does not contain any powers of $z$, hence we do not need to consider Case 3, which resolves the powers of $z$.

\section*{Non-Cyclotomic terms}

As we will see, this term becomes much more challenging than that of the powers of $z$, and we will not be able to resolve it entirely. Instead, we aim to give conditions for which the $r_{g, n, m}$ sequence is eventually periodic. Although conditions that are easy to check for users are very difficult to come by, we will provide conditions that handle many cases nonetheless.\\
\\
We first search for when $A(z)p_{u_n}(z^d) + C(z)q_{u_n}(z^d)$ is eventually coprime with $B_m(z)$, the cyclotomic factors of $B(z)$. If we can show this is true, then by the Euclidean algorithm that for sufficiently large $n$,

\begin{align*}
&\gcd(A(z)p_{u_n}(z^d) + C(z)q_{u_n}(z^d), B_m(z)q_{u_n, m}(z^d))\\
&= \gcd(A(z)p_{u_n}(z^d) + C(z)q_{u_n}(z^d), q_{u_n, m}(z^d))\\
&= \gcd(A(z), q_{u_n, m}(z^d)).
\end{align*}
This expression is far more simple to deal with, and in fact is always periodic, which we will show soon.

\begin{Lemma}
\label{lemma4.2}
Let $f(z)$ be a Mahler function satisfying (\ref{Mahler}) and let $K = \deg(f(z))$ and $f_0$ be the leading coefficient of $f(z)$. Assume
\begin{enumerate}
    \item $B_m(z)$ is coprime to $A(z^{d^t})$ for all $t \geq 0$,
    \item Every irreducible factor $B'(z)$ of $B_m(z)$ has a root $z_0$ such that $|z_0| > 1$,
    \item $\deg(C(z)) > dK + \deg(A(z))$.
    \item Let $u$ be such that $B(z_0^{d^{t+1}})C(z_0^{d^t}) \neq 0$ for $t \geq u$, then\begin{align*}
    \sum_{k = 0}^\infty \prod_{t = u}^{k-1} \frac{C(z_0^{d^{t+1}})A(z_0^{d^t})}{C(z_0^{d^t})B(z_0^{d^{t+1}})} \neq  0.
    \end{align*}
\end{enumerate}
Under these conditions we find that $\gcd(A(z)p_{u_n}(z^d) + C(z)q_{u_n}(z^d), B_m(z)) = 1$ for all sufficiently large $n$.
\end{Lemma}

While this final assumption may seem complex at first glance, it is not too difficult to check in some cases. For example, if the terms of the product are all positive for $t \geq u+1$, then it is trivially satisfied as the sum is either monotone increasing or decreasing away from zero. Additionally, as we will soon show, the ratio between the terms of the series converge to zero, hence it is possible to gain an estimate of the true value of the sum if sufficiently many terms are computed. This means it is feasible in many cases to see this condition is verified without needing compute the exact value of this infinite sum (although a computer may be necessary).

\begin{proof}
To show this, let $B'(z)$ be an irreducible non-cyclotomic factor of $B_m(z)$. To perform a proof by contradiction, suppose that for an infinite number of $n \in \N$, we have $B'(z)$ being a factor of $A(z)p_{u_n}(z^d) + C(z)q_{u_n}(z^d)$.\\
\\
Let $z_0$ be the root of $B'(z)$ from our second assumption. Then, for an infinite number of $n \in \N$, we have $A(z_0)p_{u_n}(z_0^d) + C(z_0)q_{u_n}(z_0^d) = 0$. By definition of the $G_n(z)$ polynomials, $G_{n-1}(z)p_{u_n}(z) = A(z)p_{u_{n-1}}(z^d) + C(z)q_{u_{n-1}}(z^d)$ and $G_{n-1}(z)q_{u_n}(z) = B(z)q_{u_{n-1}}(z^d)$, hence

\begin{align*}
G_{n-1}(z_0^d)(A(z_0)p_{u_n}(z_0^d) + C(z_0)q_{u_n}(z_0^d)) &= 0\\
\implies A(z_0)A(z_0^d)p_{u_{n-1}}(z_0^{d^2}) + \Big(A(z_0)C(z_0^d) + C(z_0)B(z_0^d)\Big)q_{u_{n-1}}(z_0^{d^2}) &= 0.\\
\end{align*}
We can repeat this procedure for, replacing $p_{u_{k}}(z)$ with $A(z)p_{u_{k-1}}(z^d) + C(z)q_{u_{k-1}}(z^d)$ and $q_{u_{k}}(z^d)$ with $B(z)q_{u_{k-1}}(z^d)$ for all $k \leq n$. This recursive relationship is exactly the same as that of (\ref{qkmForm}) and (\ref{pkmForm}) from several chapters past, so by the same working, we find

\begin{align*}
0 &=p_{u_0}(z_0^{d^{n+1}})\prod_{t=0}^n A(z_0^{d^t}) + q_{u_0}(z_0^{d^{n+1}})\sum_{k = 0}^{n} \Big(\prod_{t = 0}^{k-1} A(z_0^{d^t})\Big) C(z_0^{d^k}) \Big(\prod_{t = k+1}^{n}B(z_0^{d^t})\Big).\\
\end{align*}
By our second assumption, the product $\prod_{t=0}^n A(z_0^{d^t})$ is non-zero for all $n \in \N$. Additionally, $q_{u_0}(z)$ is a polynomial with a fixed finite number of zeros, so eventually $q_{u_0}(z_0^{d^{n+1}}) \neq 0$ for large enough $n$. This means for $n$ large enough,

\begin{align*}
\frac{p_{u_0}(z_0^{d^{n+1}})}{q_{u_0}(z_0^{d^{n+1}})} &= -\frac{\sum_{k = 0}^{n} \Big(\prod_{t = 0}^{k-1} A(z_0^{d^t})\Big) C(z_0^{d^k}) \Big(\prod_{t = k+1}^{n}B(z_0^{d^t})\Big)}{\prod_{t=0}^n A(z_0^{d^t})}\\
&= -\sum_{k = 0}^n \frac{C(z_0^{d^k})}{A(z_0^{d^k})} \prod_{t = k+1}^n \frac{B(z_0^{d^t})}{A(z_0^{d^t})}\\
\frac{p_{u_0}(z_0^{d^{n+1}})}{q_{u_0}(z_0^{d^{n+1}})z_0^{d^{n+1}K}} &= -\sum_{k = 0}^n \frac{C(z_0^{d^k})}{A(z_0^{d^k})z_0^{d^{n+1}K}} \prod_{t = k+1}^n \frac{B(z_0^{d^t})}{A(z_0^{d^t})}.
\end{align*}
As $p_{u_0}(z)/q_{u_0}(z)$ is a convergent of $f \in \Q((z^{-1}))$ and $\deg(f(z)) = K$,  Lemma \ref{lemma1} tells us $f(z) - p_{u_0}(z)/q_{u_0}(z)$ has a degree strictly smaller than $K$, hence as $z$ tends to $\infty$, $(f(z) - p_{u_0}(z)/q_{u_0}(z))z^{-K}$ converges to zero. Since $z_0^{d^{n+1}}$, tends to $\infty$, it follows

\begin{align*}
\lim_{n \to \infty} \frac{p_{u_0}
(z^{d^{n+1}})}{q_{u_0}(z^{d^{n+1}})z_0^{d^{n+1}K}} &= \lim_{n \to \infty} \frac{f(z_0^{d^{n+1}})}{z_0^{d^{n+1}K}} = f_0,
\end{align*}
where $f_0$ is the leading non-zero coefficient of $f(z)$. Now for the right hand side, we claim that this is not the case, and in fact given our assumptions, the above series diverges to $\infty$ as $n \to \infty$. If this is the case, then we get a clear contradiction as $n \to \infty$, hence we should only have equality for a finite number of $n$, thus $B'(z)$ is eventually coprime to $A(z)p_{u_n}(z^d) + C(z)q_{u_n}(z^d)$. As $B_m(z)$ has a finite number of irreducible factors, we eventually have $A(z)p_{u_n}(z^d) + C(z)q_{u_n}(z^d)$ coprime to $B_m(z)$, as required. Note that divergence to $\infty$ is essential as opposed to simply not approaching a limit, as the expressions may be equal for infinitely many $n$ if we don't have divergence to $\infty$. Now we show that the series does diverge to $\infty$ by contradiction, so assume it remains bounded as $n \to \infty$.\\
\\
Without loss of generality, assume that each term of the series is non-zero, as we may simply discard the initial few terms if this is the case and by the following logic, we will still see the series diverge. In terms of assumption (d), this means $u = 0$. As $B(z)$ and $C(z)$ only contain finitely many roots and $z_0^{d^k}$ tends to $\infty$ as $k \to \infty$, we will be discarding only finitely many terms. From the original Mahler equation (\ref{Mahler}) and the assumption that $\deg(A(z)) + dK < \deg(C(z))$, we find
\begin{align*}
\deg(B(z)f(z)) = \deg(A(z)f(z^d) + C(z)) = \deg(C(z)),
\end{align*}
hence $\deg(B(z)) + K = \deg(C(z))$. Using this, we can compare the ratio between consecutive terms of the sum for a fixed $n$. As for sufficiently large $k$ the terms of the sum will be non-zero,

\begin{align*}
&\Bigg| \frac{C(z_0^{d^{k+1}})}{A(z_0^{d^{k+1}})z_0^{d^{n+1}K}} \prod_{t = k+2}^n \frac{B(z_0^{d^t})}{A(z_0^{d^t})}\Bigg|\Big/ \Bigg|\frac{C(z_0^{d^k})}{A(z_0^{d^k})z_0^{d^{n+1}K}} \prod_{t = k+1}^n \frac{B(z_0^{d^t})}{A(z_0^{d^t})}\Bigg|\\
&= \Bigg|\frac{C(z_0^{d^{k+1}})A(z_0^{d^k})}{C(z_0^{d^k})B(z_0^{d^{k+1}})}\Bigg| =  \Bigg|\frac{C(z_0^{d^{k}d})A(z_0^{d^k})}{C(z_0^{d^k})B(z_0^{d^{k}d})}\Bigg|.
\end{align*}
Interpreting these polynomials as polynomials of $z_0^{d^k}$, the degree of the numerator minus the degree of the denominator is

\begin{align*}
(dr_c + r_a) - (r_c + dr_b) &= d(r_c-r_b) + r_a-r_c < dK - dK = 0.
\end{align*}
It follows that as $k \to \infty$, we see $z_0^{d^k} \to \infty$ and the asymptotic growth of the denominator is larger than that of the numerator, hence the ratio converges to zero. Furthermore, the ratio between terms is entirely independent of $n$. This implies we may write the series in the form,

\begin{align*}
\sum_{k = 0}^n A_n R_k &= A_n \sum_{k=0}^n R_k,
\end{align*}
where $A_n$ denotes the first non-zero term of the series, and $R_k$ is the the product of the ratio between the terms of the original series, and independent of $n$.

\begin{align*}
A_n = \frac{C(z_0)}{A(z_0)z_0^{d^{n+1}K}} \prod_{t = 1}^n \frac{B(z_0^{d^t})}{A(z_0^{d^t})} \quad R_k = \prod_{t = 0}^{k-1} \frac{C(z_0^{d^{t+1}})A(z_0^{d^t})}{C(z_0^{d^t})B(z_0^{d^{t+1}})}.
\end{align*}
Now observe that $A_n$ diverges to $\infty$. Indeed, first notice that 

\begin{align*}
d^{n+1} &= d + (d-1)(d + d^2 + ... + d^n)\\
\implies z_0^{d^{n+1}K} &= z_0^{dK} \prod_{t = 1}^n z_0^{(d-1)d^t K}.
\end{align*}
It follows we can rewrite $A_n$,
\begin{align*}
A_n = \frac{C(z_0)}{A(z_0)z_0^{dK}} \prod_{t = 1}^n \frac{B(z_0^{d^t})}{A(z_0^{d^t})z^{(d-1)d^tK}}.
\end{align*}
Now investigating the product, we find each term diverges as $t$ increases. This is clear if we investigate each term as a polynomial of $z_0^{d^t}$, as then the degree of each term is $r_b - r_a - (d-1)K = r_c-K-r_a-dK+K > 0$. It follows the terms of the infinite product diverge, and hence so must the expression for $A_n$. If $A_n\sum_{k=0}^n R_k$ is to be bounded over $n$, we must have the sum over $R_k$ equal to zero as $n \to \infty$. That is,
\begin{align*}
\sum_{k = 0}^\infty \prod_{t = 0}^{k-1} \frac{C(z_0^{d^{t+1}})A(z_0^{d^t})}{C(z_0^{d^t})B(z_0^{d^{t+1}})} =  0.
\end{align*}
This however is a direct contradiction to our initial assumption, completing the proof. It should be noted that this series does indeed converge, as by applying the ratio test, we simply get the ratio between terms is the same as the previous ratio found which converges to zero as we have already shown, thus the series converges absolutely.
\end{proof}

Although some of the assumptions of this Lemma are strong, it is certainly possible to loosen the assumptions of the Lemma. For example, suppose all the assumptions of Lemma \ref{lemma4.2} hold with the exception of (c) and (d), where we instead have $\deg(C(z)) < dK + \deg(A(z))$. Then by performing the same proof by contradiction, we reach the same series result,
\begin{align*}
\frac{p_{u_0}(z_0^{d^{n+1}})}{q_{u_0}(z_0^{d^{n+1}})z_0^{d^{n+1}K}} &= -\sum_{k = 0}^n \frac{C(z_0^{d^k})}{A(z_0^{d^k})z_0^{d^{n+1}K}} \prod_{t = k+1}^n \frac{B(z_0^{d^t})}{A(z_0^{d^t})},
\end{align*}
where, as $n \to \infty$, the left hand converges to the same $f_0$, the leading coefficient of $f(z)$. We claim the series also converges as $n \to \infty$, however there is no guarantee that it converges to $f_0$. For now, assume the series does in fact converge as $n \to \infty$, and note that its limit depends entirely on $A(z), B(z)$ $C(z)$ and $d$. If we were to expect equality always, then this suggests that $f_0$ can be expressed entirely in terms of $A(z), B(z), C(z)$ and $d$, however this is false.\\
\\
If we consider the Mahler equation (\ref{Mahler}) and try to solve for $f_0$. Then the leading terms suggest $\alpha f_0 = \beta f_0$, where $\alpha$ and $\beta$ are the leading terms of $A(z)$ and $B(z)$ respectively. As $f_0 \neq 0$, we have $\alpha = \beta$ and this equation provides no information about the value of $f_0$. If we consider the next coefficient, we find $f_1$, the next coefficient of $f(z)$, as a function of $f_0$ and the other terms of $A(z), B(z)$ and $C(z)$. This continues for each coefficient, leading to all coefficients depending on $f_0$ which may be chosen arbitrarily. It follows that as $n \to \infty$, we have a contradiction if $f_0$ is such that

\begin{align*}
f_0 \neq -\lim_{n \to \infty} \sum_{k = 0}^n \frac{C(z_0^{d^k})}{A(z_0^{d^k})z_0^{d^{n+1}K}} \prod_{t = k+1}^n \frac{B(z_0^{d^t})}{A(z_0^{d^t})}.
\end{align*}
Sadly this condition is not as simple to check as the ones of the Lemma \ref{lemma4.2}, however this applies to all except one Mahler function for a given $A(z), B(z), C(z)$ and $d$ and thus is not too restrictive if considering all Mahler functions. Now showing that the series does indeed converge, for simplicity assume $B(z_0^{d^t}) \neq 0$ for $t \geq 1$, as if this is not the case, then as $B$ is a polynomial, this is true for all $t$ large enough, hence some finite number of the leading terms of the sum are zero and we analyse the rest. With this we can factor out the product terms giving,

\begin{align*}
&\sum_{k = 0}^n \frac{C(z_0^{d^k})}{A(z_0^{d^k})z_0^{d^{n+1}K}} \prod_{t = k+1}^n \frac{B(z_0^{d^t})}{A(z_0^{d^t})} = \sum_{k = 0}^n \frac{C(z_0^{d^k})}{A(z_0^{d^k})z_0^{d^{k+1}K}} \prod_{t = k+1}^n \frac{B(z_0^{d^t})}{A(z_0^{d^t})z_0^{(d-1)d^{t}K}}\\
&= \Big(\prod_{t = 1}^n \frac{B(z_0^{d^t})}{A(z_0^{d^t})z_0^{(d-1)d^{t}K}}\Big)\sum_{k = 0}^n \frac{C(z_0^{d^k})}{B(z_0^{d^k})z_0^{d^{k}K}} \prod_{t = 1}^{k-1}\frac{A(z_0^{d^t})z_0^{(d-1)d^{t}K}}{B(z_0^{d^t})}.
\end{align*}
We now need to check the convergence of the infinite product and the series (with terms that no longer depend on $n$). For the series, we use a ratio test. The ratio between the $k+1$ and the $k$-th term is
\begin{align*}
&\frac{C(z_0^{d^{k+1}})A(z_0^{d^k})z_0^{(d-1)d^kK}}{B(z_0^{d^{k+1}})z_0^{d^{k+1}K}B(z_0^{d^k})}\times \frac{B(z_0^{d^k})z_0^{d^kK}}{C(z_0^{d^k})}\\
&= \frac{C(z_0^{d^{k+1}})A(z_0^{d^k})}{B(z_0^{d^{k+1}})C(z_0^{d^{k}})}.
\end{align*}
Interpreting these as polynomials of $z_0^{d^k}$ as usual, we have the degree of the numerator minus the degree of the denominator being $dr_c + r_a - dr_b - r_c$. By comparing degrees in (\ref{Mahler}), we find $r_b + K = r_a + dK > r_c$, hence $dr_c + r_a - dr_b - r_c < (d-1)(r_a + dK) + r_a - d(r_a + (d-1)K) = 0$. This implies the numerator grows slower asymptotically than the denominator, and the ratio converges to zero as $k \to \infty$, hence the series converges absolutely. We also claim the infinite product converges absolutely by checking the sum

\begin{align*}
\sum_{t = 1}^ \infty \Big| \frac{B(z_0^{d^t})}{A(z_0^{d^t})z_0^{(d-1)d^tK}} - 1\Big|  = \sum_{t = 1}^ \infty \Big| \frac{B(z_0^{d^t}) - A(z_0^{d^t})z_0^{(d-1)d^tK}}{A(z_0^{d^t})z_0^{(d-1)d^tK}}\Big| < \infty.
\end{align*}
Note that we have already seen by comparing the leading coefficients for the original Mahler equation (\ref{Mahler}), $\alpha = \beta$. As $r_b + K = r_a +dK$, the leading terms in the sum $B(z) - A(z)z^{(d-1)K}$ cancel each other. We can write $P(z) = B(z) - A(z)z^{(d-1)K}$ and $Q(z) = A(z)z^{(d-1)K}$ and note $\deg(P(z)) < \deg(Q(z))$. Applying the ratio test to $\sum_{t = 0}^\infty |\frac{P(z_0^{d^t})}{Q(z_0^{d^t})}|$, we see the ratio between terms as
\begin{align*}
\frac{P(z_0^{d^{t+1}})Q(z_0^{d^t})}{Q(z_0^{d^{t+1}})P(z_0^{d^t})}.
\end{align*}
As a polynomial of $z_0^{d^t}$, the degree of the numerator is larger than that of the denominator, hence this ratio converges to zero as $t \to \infty$, thus the series converges.\\
\\

We have successfully reduced $G_{n, m}(z)$ down to $\gcd(A(z), q_{u_n, m}(z^d))$ for all $n$ large enough. We claim the degree of this is periodic in all cases, which we can prove using an argument similar to that of Badziahin \cite{badziahin}. We make note of a Lemma first.

\begin{Lemma}
\label{lemma4.1}
Let $A(z), D(z) \in \Q[z]$ be polynomials such that any root $z_0$ of $D(z)$ is such that $|z_0| \in (0, \infty)\setminus \{1\}$. Then there exists some $n_0 \in \N$ such that for all $n \geq n_0$, $\gcd(A(z), D(z^{d^n})) = 1$.
\end{Lemma}

\begin{proof}
If this is not the case, then $A(z)$ having finitely many roots implies at least one root of $A(z)$ is a root of $D(z^{d^n})$ for infinitely many $n$. Let this root be $z_0$. If $z_0 = 0$ or $|z_0| = 1$, the by noticing that for all $n \geq 0$, $z_0^{d^n} = 0$ or $|z_0^{d^n}| = 1$, so $z_0$ cannot be a root of $D(z^{d^n})$, a contradiction.\\
\\
Now suppose $|z_0| \in (0, \infty) \setminus \{1\}$. As $D(z)$ has finitely many roots, there exists a $\alpha \in (0, 1)$ and $\beta \in (1, \infty)$ such that $D(z) = 0 \implies \alpha < |z| < \beta$. Replacing $z$ with $z^{d^n}$, we have $D(z^{d^n}) = 0 \implies \alpha^{1/d^n} < |z| < \beta^{1/d^n}$. Simply take $n_0$ large enough such that $z_0 \notin (\alpha^{1/d^{n_0}}, \beta^{1/d^{n_0}})$, implying $z_0$ cannot be a root of $D(z^{d^n})$ for $n \geq n_0$, a contradiction.
\end{proof}

As we eventually have $G_{n, m}(z) = \gcd(A(z), q_{u_n, m}(z^d))$ and we have not imposed any conditions on the convergents of $f$, we may simply start studying the primitive gap sequence at the point we have equality. It follows that it's enough for us to prove $r_{g, n, m}$ is eventually periodic if this equality is expected for all $n \geq 0$ by simply adjusting where we start looking at the primitive gap sequence. By definition, we have $q_{u_n, m}(z)$ dividing $B_m(z)q_{u_{n-1}, m}(z^d)$ for all $n \geq 1$, thus it follows $q_{u_n}(z)$ is some factor of  $q_{u_0, m}(z^{d^n})\prod_{k=0}^{n-1} B(z^{d^k})$. Accounting for the removed factors, we can write $q_{u_n, m}(z^d)$ in the form

\begin{align*}
q_{u_n, m}(z^d) = q'_{n}(z) \prod_{k=1}^{n}B'_{k, n}(z).
\end{align*}
where $B'_{k, n}(z)$ divides $B_m(z^{d^k})$. And $q'_{n}(z)$ divides $q_{u_n, m}(z^{d^{n+1}})$. By Lemma \ref{lemma4.1}, we note that by setting $D(z) = B_m(z)$ and $D(z) = q_{u_0, m}(z)$, we satisfy all the required assumptions, hence they will not contribute to $G_{n, m}(z)$ for sufficiently large $n$. This means for large $n$, there exists some $k_0$ independent of $n$ such that

\begin{align*}
\gcd(A(z), q_{u_n, m}(z^d)) &= \gcd(A(z), \prod_{k = 1}^{k_0} B'_{k, n}(z)).
\end{align*}
We now apply a combinatorial argument. For sufficiently large $n$, As each $B_m(z^{d^k})$ has a finite number of factors for $1 \leq k \leq k_0$, there exists some $n_1 < n_2$ such that $B'_{n_1, k}(z) = B'_{n_2, k}(z)$ for all $1 \leq k \leq k_0$ and thus $G_{n_1, m}(z) = G_{n_2, m}(z)$. As $q_{u_{n_1+1}, m}(z) = B_m(z)q_{u_{n_1}, m}(z^d)/G_{n_1, m}(z)$ and similarly for $n_2$, it follows that $B'_{n_1+1, k}(z)$ depends entirely on the previous $B'_{n_1, j}(z) = B'_{n_2, j}(z)$ for $1 \leq j \leq k_0$. This means $B'_{n_1+1, k}(z) = B'_{n_2+1, k}(z)$ for all $1 \leq k \leq k_0$. It follows that the $B'_{n, k}(z)$ are periodic, and hence so must $G_{n, m}(z)$ and $r_{g, n, m}$ for sufficiently large $n$.\\
\\

\section*{Cyclotomic terms}

This is most challenging part of showing the $r_{g, n}$ sequence is periodic. In the case $C(z) = 0$, a significant majority of the proof Badziahin gives to show $r_{g, n}$ is periodic is in reference to this case, with the previous two cases being relatively simple in comparison. This is due to the fact that there are many distinct cyclotomic polynomials to keep track of and the fact that we cannot resort to limiting techniques as the roots do not tend to $\infty$. To resolve this, we reduce to the following case:

\begin{Lemma}
\label{lemma5.4}
Suppose that there exists some $n_0 \in \N$ such that $A(z)p_{u_n}(z^d) + C(z)q_{u_n}(z^d)$ is coprime to $B_c(z)$ for $n \geq n_0$. Then the $r_{g, n, c}$ sequence is eventually periodic.
\end{Lemma}

Sadly, the conditions that induce the assumption of this Lemma are not clear other than trivial ones such as $B_c(z) = 1$ as it requires us to map the roots of $A(z)p_{u_n}(z^d) + C(z)q_{u_n}(z^d)$ with precision about the unit circle, however it is necessary if we wish to solve this problem in a tractable manner while maintaining some level of generality. This assumption is certainly not expected to hold in general, however restricting to this case allows us to apply the Euclidean algorithm as we have done before, reducing to an expression similar to that of the non-cyclotomic case.

\begin{align*}
G_{n, c}(z) &= \gcd(A(z)p_{u_n}(z^d) + C(z)q_{u_n}(z^d), B_c(z)q_{u_n, c}(z^d))\\
&= \gcd(A(z)p_{u_n}(z^d), q_{u_n, c}(z^d))\\
&= \gcd(A_c(z), q_{u_n, c}(z^d)),
\end{align*}
where $A_c(z)$ denotes the cyclotomic factors of $A(z)$. As there is no addition term, we can apply an argument similar to that of Badziahin.\\
\\
One definition of a cyclotomic polynomial is an irreducible monic factor of $z^k - 1$ for some $k \geq 1$. It's possible to show that the cyclotomic polynomials can be indexed as $\Phi_n(z), n \geq 1$ where

\begin{align*}
\Phi_n(z) &= \prod_{\substack{1 \leq k \leq n\\ \gcd(i, n) = 1}} (z - e^{2i\pi \frac{k}{n}}).
\end{align*}
Before we prove Lemma \ref{lemma5.4}, we must introduce some additional Lemmas and notation, given by Badziahin \cite{badziahin}. For any two positive integers $m, n$, let $r(m, n)$ be the largest divisor of $m$ which is coprime with $n$. Also let $s(m, n) = \frac{m}{r(m, n)}$. In terms of the prime factorisation of $m$, $r(m, n)$ contains the prime factors of $m$ that are not prime factors of $n$, and $s(m, n)$ contains the prime factors of $m$ that are factors of $n$.
\begin{Lemma}
\label{lemma5.5}
For any positive integers $n, d$,
\begin{align*}
\Phi_n(z^d) &= \prod_{r | r(d, n)} \Phi_{rns(d, n)}(z).
\end{align*}
\end{Lemma}

\begin{proof}[Proof of Lemma \ref{lemma5.5}]
Any root of $\Phi_n(z)$ is of the form $\xi_n^k$ where $\xi_n = e^{2i \pi / n}$ and $\gcd(k, n) = 1$. It follows that any root $\xi$ of $\Phi_n(z^d)$ satisfies $\xi^d = \xi_n^k = \xi_{nd}^{kd + ndt}$ where $0 \leq t \leq d-1$. Thus the roots of $\Phi_n(z^d)$ are
\begin{align*}
\xi = \xi_{nd}^{k + nt} \quad 0 \leq t \leq d-1 \text{ and } 1 \leq k \leq d-1.
\end{align*}
where $\gcd(k, n) = 1$ and hence $\gcd(k+nt, n) = 1$. Note that $nt+k$ must run through all integers between 1 and $nd$ that are coprime to $n$. Denote this set $N$, and split $N$ into disjoint subsets based on each elements common factors with $d$.
\begin{align*}
N_t &= \{x \in N | \gcd(x, d) = t\}.
\end{align*}
The $N_t$ can only be non-empty when $t | d$ and $\gcd(t, n) = 1$ (as otherwise $\gcd(x, n) \neq 1$, a contradiction to the definition of $N$). This is the same as $t | r(d, n)$. Fix a value of $t$ such that $N_t$ is non-empty, and let $r = r(d, n)/t$. For any $x \in N_t$, $x$ is divisible by $t$ hence $\xi_{nd}^x = \xi_{nd/t}^{x/t}$. Also note that $\gcd(x/t, nd/t) = \gcd(x, nd)/t = \gcd(x, d)/t = t/t = 1$, so $\xi_{nd/t}^{x/t}$ is a root of $\Phi_{nd/t}(z) = \Phi_{rns(d, n)}(z)$. Putting this together for all $t$ and reversing the product index for clarity,
\begin{align*}
\Phi_{n}(z^d) = \prod_{t | r(d, n)} \Phi_{rns(d, n)}(z) = \prod_{r | r(d, n)} \Phi_{rns(d, n)}(z).
\end{align*}
\end{proof}

\begin{Lemma}
\label{lemma5.7}
Fix $n \in \N$ and let $r = r(n, d)$. Then there exists a $m \in \N$ such that $\Phi_n(z) | \Phi_r(z^{d^m})$.
\end{Lemma}

\begin{proof}
Write $n = rs$ where $s = s(n, d)$. We will show the result holds true for any value of $s \geq 1$ for a fixed value of $n$ by induction on $s$. If $s = 1$, the statement is obvious as $n = r$ and we may choose $m = 0$.\\
\\
Let the \textit{radical} of an integer $s$, denoted $\rad(s)$, be the product of each of the prime factors of $s$. As $r(n, d)$ contains all the factors of $n$ that are not factors of $d$, it follows $s(n, d)$ contains all the factors of $n$ that are factors of $d$, and hence $\rad(s) | \rad(d)$. Fix some $S \in \N$ such that $\rad(S) | \rad(d)$ and suppose the statement holds true for all $s < S$. For the inductive step, we must show the statement holds for $s(n, d) = S$. Write the prime factorisations of $S$ and $d$ as
\begin{align*}
S = p_1^{\beta_1}p_2^{\beta_2}...p_k^{\beta_k}p_{k+1}^{\beta_{k+1}}...p_{k+l}^{\beta_{k+l}} \quad d = p_1^{\alpha_1}...p_{k+l}^{\alpha_{k+l}}
\end{align*}
where the $\alpha_i, \beta_i$ are integers such that $0 \leq \beta_i < \alpha_i$ for $i \leq k$ and $0 < \alpha_i \leq \beta_i$ for $k+1 \leq i \leq k+l$. Now fix $s = p_{k+1}^{\beta_{k+1}-\alpha_{k+1}}p_{k+1}^{\beta_{k+2}-\alpha_{k+2}}...p_{k+l}^{\beta_{k+l}-\alpha_{k+l}}$, which is certainly less than $S$. We also find $r(d, s)$ is a multiple of $p_1^{\beta_1}...p_k^{\beta_k}$, thus
\begin{align*}
p_1^{\beta_1}...p_k^{\beta_k} s s(d, s) &= p_1^{\beta_1}...p_k^{\beta_k} \times p_{k+1}^{\beta_{k+1}-\alpha_{k+1}}...p_{k+l}^{\beta_{k+l}-\alpha_{k+l}} \times p_{k+1}^{\alpha_{k+1}}...p_{k+l}^{\alpha_{k+l}}\\
&= p_1^{\beta_1}...p_{k+l}^{\beta_{k+1}} = S.
\end{align*}
It immediately follows by Lemma \ref{lemma5.5} that $\Phi_S(z)$ is a factor of $\Phi_s(z^d)$. By the inductive hypothesis, $\Phi_s(z) | \Phi_r(z^{d^m})$ for some $m$, hence $\Phi_S(z) | \Phi_r(z^{d^{m+1}})$. This completes the inductive step for $S$, and hence the result holds true regardless of what $s(n, d)$ is, and thus for all $n$.

\end{proof}

For any $P(z), Q(z) \in \Z[z]$ with $P(z)$ non-constant, let $\sigma(P(z), Q(z))$ be the largest integer $n$ such that $P(z)^n | Q(z)$. With this, we have the last required Lemma before the main proof.

\begin{Lemma}
\label{lemma5.8}
Fix $P(z) \in \Z[z]$ and $k \in \N$, $k \geq 0$. Then $\sigma(\Phi_k(z), P(z^{d^m}))$ is bounded from above for all $m \in \N$.
\end{Lemma}

\begin{proof}
Write $k = rs$ where $r = r(k, d)$ and $s = s(k, d)$. By definition, $r$ is coprime with $d$ and $\rad(s) | \rad(d)$. Now write $P(z) = P_r(z)Q(z)$ where $P_r(z)$ contains all the roots of $P(z)$ that are roots of unity with some degree $k'$ where $r(k', d) = r$. As the roots of the cyclotomic polynomial $\Phi_n(z)$ are the roots of unity of degree $n$, we can write $P_r(z)$ in the form
\begin{align*}
P_r(z) &= \prod_{\rad(p) | \rad(d)} \Phi_{rp}(z)^{\alpha(p)}.
\end{align*}
Note that for any cyclotomic factor $\Phi_t(z)$ of $\Phi_n(z^d)$, Lemma \ref{lemma5.5} suggests $r(t, d) = r(ns(d, n), d) = r(n, d)$ i.e. the $r$-value is preserved. As the cyclotomic factors $\Phi_t(z)$ of $Q(z)$ have $r(t, d) \neq r$, it follows that this also holds for the factors of $Q(z^{d^m})$, and hence $Q(z^{d^m})$ is always coprime with $\Phi_k(z)$. This means $\sigma(\Phi_k(z), P(z^{d^m})) = \sigma(\Phi_k(z), P_r(z^{d^m}))$.\\
\\
We will now induct on the values of $s$ in a similar manner to the proof of the previous Lemma. If $s = 1$, applying Lemma \ref{lemma5.5} gives
\begin{align*}
P_r(z^{d^m}) &= \prod_{\rad(p)|\rad(d)} \Phi_{rp}(z^{d^m})^{\alpha(p)}\\
&= \prod_{\rad(p)|\rad(d)}\prod_{t|r(d, rp)^m}\Phi_{trps(d, rp)^m}(z)^{\alpha(p)}\\
&= \prod_{\rad(p)|\rad(d)}\prod_{t|r(d, p)^m}\Phi_{trps(d, p)^m}(z)^{\alpha(p)}.
\end{align*}
As $k = r$, the only factor to possibly have an effect is when $t = p = 1$ so $s(d, p) = 1$, hence
\begin{align*}
\sigma(\Phi_k(z), P_r(z^{d^m})) = \sigma(\Phi_r(z), \Phi_r(z)^{\alpha(1)}) = \alpha(1),
\end{align*}
which is clearly bounded independently of $m$. Now fix $S \in \N$ such that $\rad(S) | \rad(d)$, and assume the lemma holds for all $s < S$ for our value of $k$. Using Lemma \ref{lemma5.5} we note the degree of any factor of $\Phi_n(z^d)$ is larger than $n$, thus by simply considering the $\Phi_{rs}(z)$ with $s < S$, we have the upper bound
\begin{align*}
\sigma(\Phi_{rS}(z), P_r(z^{d^m})) \leq \sigma(\Phi_{rS}(z), \Phi_{rS}(z^{d^m})^{\alpha(S)}) + \sum_{s | S, s < S}\sigma(\Phi_{rs}(z), P_r(z^{d^{m-1}})).
\end{align*}
The first term is simply $\alpha(S)$, however the summation term sums over $s < S$, hence by the inductive hypothesis, each term is bounded by some constant independent of $m$, and so must the sum. This forms an upper bound independent of $m$, and finishes the inductive step and the proof.
\end{proof}
With these tools, we are now ready to prove Lemma \ref{lemma5.4} which will conclude this final chapter.

\begin{proof}[Proof of Lemma \ref{lemma5.4}]
Similar to the justification of the cyclotomic case, it is safe for us to assume $A(z)p_{u_n}(z^d) + C(z)q_{u_n}(z^d)$ is always periodic to $B_c(z)$ and

\begin{align*}
r_{g, n, c} = \deg(\gcd(A_c(z), q_{u_n, c}(z^d)).
\end{align*}
We note from Lemma \ref{lemma5.5}, any divisor $\Phi_k(z)$ of $\Phi_n(z^{d^m})$ satisfies $r(k, d) = r(ns(d, n), d) = r(n, d)$. By definition, $q_{u_n, c}(z) |q_{u_0, c}(z^{d^m}) \prod_{k = 0}^{n-1} B_c(z^{d^k})$, hence we can split $q_{u_n, c}(z^d)$ into a product of cyclotomic factors $\Phi_k(z)$ grouped by the value of $r(k, d)$. These factors will not interact later as the r-value is preserved, hence we can split $r_{g, n, c}$ into a sum and $G_{n, c}(z)$ into a product,

\begin{align*}
r_{g, n, c} &= \sum_{r \in \N, \gcd(d, r) = 1} r_{g, n, c, r}\\
G_{n, c}(z) &= \prod_{r \in \N, \gcd(d, r) = 1} G_{n, c, r}(z).
\end{align*}
Furthermore, as $A_c(z)$ has finitely many cyclotomic factors, we have a finite number of these sequences having non-zero terms, hence it is enough for us to show $r_{g, n, c, r}$ is eventually periodic for each $r \in \N, \gcd(d, r) = 1$. We will do this across three cases.\\
\\
\underline{Case 1:} Both $B_{c, r}(z)$, and  $q_{u_0, c, r}(z)$ do not contain the factor $\Phi_r(z)$. Let $\Phi_{rp}(z)$ be some factor of $B_{c, r}(z)$ or $q_{u_0, c, r}(z)$ where $\rad(p) | \rad(d)$ and $p > 1$. Now note by Lemma \ref{lemma5.5}, any factors $\Phi_k(z)$ of $\Phi_{rp}(z^{d^n})$ have $rps(d, rp)^n | k$. Note that $s(d, rp) > 1$ as $\rad(p) | \rad(d)$ and $p > 1$. This means that as $n \to \infty$, we also need $k \to \infty$ for $\Phi_k(z)$ to divide $\Phi_{rp}(z^{d^n})$. It follows that there exists some $n_0 \in \N$ such that for $n \geq n_0$,
\begin{align*}
\gcd(A_{c, r}(z), B_{c, r}(z^{d^n})) = \gcd(A_{c, r}(z), q_{u_0, c, r}(z^{d^n})) = 1.
\end{align*}
This is analogous to the case of non-cyclotomic polynomials, as once we write $q_{u_n, c, r}(z)$ in the form
\begin{align*}
q_{u_n, c, r}(z^d) = q'_{n}(z) \prod_{k=1}^{n}B'_{k, n}(z),
\end{align*}
where $q'_n(z) | q_{u_0, c, r}(z^{d^{n+1}})$ and $B'_{k, n}(z) | B_{c, r}(z^{d^k})$, we find there exists a $n_0 \in \N$ such that for all $n \geq n_0$,
\begin{align*}
r_{g, n, c, r} &= \gcd(A_{c, r}(z), \prod_{k=1}^{n_0}B'_{k, n}(z)).
\end{align*}
Applying a very similar combinatorial argument to that of the non-cyclotomic case will result in this sequence being eventually periodic.\\
\\
\underline{Case 2}: $\Phi_r(z)$ has a factor of $q_{u_0, c, r}(z)$, but not $B_{c, r}(z)$. We may then write
\begin{align*}
A_{r, c}(z) &= \prod_{i = 1}^m \Phi_{rp_i}(z).
\end{align*}
where the $p_i$ are some collection of integers. Let $S$ be the set of divisors of the $p_i$. That is, $S = \{s \in \N :\, s | p_i$ for some $1 \leq i \leq m\}$. Finally, let the elements of $S$ be enumerated as $s_1, s_2, ..., s_N$ and

\begin{align*}
\sigma_{s_i, n} := \sigma(\Phi_{rs_i}(z), q_{u_n, c, r}(z^d)), \quad \Sigma_n := (\sigma_{s_1, n}, ..., \sigma_{s_N, n}).
\end{align*}
As each $p_i \in S$ by definition, we see by definition that $\gcd(A_{r, c}(z), q_{u_n, c, r}(z^d))$ can be computed exactly given $\Sigma_n$, and thus so can $q_{u_{n+1}, c, r}(z^d)$. Additionally, noting that $\Phi_k(z)$ can only be a factor of $\Phi_n(z^d)$ if $k$ is a multiple of $n$, we have $\Sigma_{ n+1}$ entirely dependent on $\Sigma_n$, $A_{c, r}(z)$ and $B_{c, r}(z)$. If we assume that each element of $\Sigma_n$ is bounded independently of $n$, then clearly we must have $\Sigma_{n_0} = \Sigma_{n_1}$ for some $n_0 \neq n_1$ as there are only finitely many values $\Sigma_n$ could take. Due to the dependence on the previous element only, we have $\Sigma_{n_0+1} = \Sigma_{n_1+1}$, and so on, implying $\Sigma_n$ is eventually periodic, and hence so must $\gcd(A_{r, c}(z), q_{u_n, c, r}(z^d))$. It remains to show each element of $\sigma(\Phi_{rs_i}(z), q_{u_n, c, r}(z^d))$ is bounded independently of $n$.\\
\\
We recall
\begin{align*}
q_{u_n, c, r}(z^d) | q_{u_0, c, r}(z^{d^{n+1}})\prod_{t = 1}^n B_{c, r}(z^{d^t}).
\end{align*}
By applying Lemma \ref{lemma5.8}, we have $\sigma(\Phi_{rs_i}(z), q_{u_0, c, r}(z^{d^{n+1}}))$ bounded from above by some constant independent of $n$. Furthermore, by applying the ideas of Case 1 here, as $\Phi_r(z)$ does not divide $B_{r, c}(z)$, there exists a $n_0$ such that $\Phi_{rs_i}(z)$ is coprime with $B_{r, c}(z^{d^n})$ for $n \geq n_0$. Thus,
\begin{align*}
\sigma(\Phi_{rs_i}(z), q_{u_n, c, r}(z^d)) &\leq \sigma(\Phi_{rs_i}(z), q_{u_0, c, r}(z^{d^{n+1}})) + \sum_{t=1}^{n_0} \sigma(\Phi_{rs_i}(z), B_{c, r}(z^{d^t})),
\end{align*}
which is bounded from above independently of $n$ for each $s_i$. This finishes this case.\\
\\
\underline{Case 3:} $\Phi_r(z)$ divides $B_{c, r}(z)$ but not $A_{c, r}(z)$. We perform a similar set up to Case 2, with the same $p_i$, set of divisors $S$ and $\Sigma_n$ vector where $\Sigma_{n+1}$ can be determined using $\Sigma_n$ and $\gcd(A_{c, r}(z), q_{u_n, c, r}(z^d))$ can be computed from the $\Sigma_n$ sequence. Due to the recursive relationship
\begin{align*}
q_{u_{n+1}, c, r}(z) &= \frac{B_{c, r}(z)q_{u_n, c, r}(z^d)}{\gcd(A_{c, r}(z), q_{u_n, c, r}(z^d))},
\end{align*}
the $\gcd$ will not eliminate the $\Phi_r(z)$ factors from $q_{u_n, c, r}(z^d)$ as $A_{c, r}(z)$ does not have any of these factors. As $\Phi_r(z) | \Phi_r(z^d)$, we must have $\sigma(\Phi_r(z), q_{u_n, c, r}(z^d)) \geq n$. When computing $\Sigma_{n+1}$ given $\Sigma_n$, we imagine this as removing the cyclotomic factors of $A_{c, r}(z)$ from each of the corresponding elements in $\Sigma_n$ where possible, adding the cyclotomic factors of $B_{c, r}(z)$ then performing a "redistribution" when computing $q_{u_{n+1}, c, r}(z^d)$ by applying Lemma \ref{lemma5.5} to each factor of $q_{u_{n+1}, c, r}(z)$. As $\sigma(\Phi_r(z), q_{u_n, c, r}(z^d)) \geq n$, the redistribution will cause many elements of $\Sigma_n$ to increase larger than the maximum removal of the cyclotomic factors of $A_{c, r}(z)$, thus we will expect a variety of elements of $\Sigma_n$ to eventually monotonically increase over $n$. Lemma \ref{lemma5.7} suggests that this occurs to all elements in $\Sigma_n$, hence it naturally follows $\gcd(A_{c, r}(z), q_{u_n, c, r}(z^d)) = A_{c, r}(z)$ eventually and $r_{g, n, c, r}$ is eventually constant.\\
\\
One final case we have left out is $B_{c, r}(z)$ and $A_{c, r}(z)$ both contain a factor of $\Phi_r(z)$, however it turns out that this will form a contradiction to our original assumption. Indeed, we have assumed $A(z)p_{u_{{n_0+1}}}(z^d) + C(z)q_{u_{n_0+1}}(z^d)$ is coprime to $B_c(z)$. This implies $q_{u_{n_0+1}}(z^d)$ cannot contain a factor of $\Phi_r(z)$, however we have
\begin{align*}
q_{u_{n_0+1}, c, r}(z^d) &= B_{c, r}(z^d)\frac{q_{u_{n_0}, c, r}(z^{d^2})}{\gcd(A_{c, r}(z^d), q_{u_{n_0}, c, r}(z^{d^2}))}.
\end{align*}
As $B_{c, r}(z)$ contains $\Phi_r(z)$, so does $B_{c, r}(z^d)$ by Lemma \ref{lemma5.5}, 
thus $q_{u_{n_0+1}}(z^d)$ contains a factor of $\Phi_r(z)$. This is a clear contradiction and finishes the proof.



\end{proof}

Although this proof seemingly implies the assumptions of the Lemma will never hold if there exists a $n$ coprime with $d$ such that $\Phi_n(z)$ divides $A(z)$ and $B(z)$, this can be avoided via a transformation similar to that of the powers of $z$ case. Indeed, without loss of generality we may assume
\begin{align*}
f(z) &= \frac{A(z)\Phi_n(z)f(z^d) + C(z)}{B(z)\Phi_n(z)}\\
\implies f(z)\Phi_n(z) &= \frac{A(z)\Phi_n(z)f(z^d) + C(z)}{B(z)}.
\end{align*}
Note that by Lemma \ref{lemma5.5} we have $\Phi_n(z^d) = \prod_{r | d} \Phi_{rn}(z)$, hence

\begin{align*}
f(z)\Phi_n(z) &= \frac{A(z)f(z^d)\Phi_n(z^d) + C(z)\prod_{r | d, r \neq 1}\Phi_{rn}(z)}{B(z)\prod_{r | d, r \neq 1}\Phi_{rn}(z)}.
\end{align*}
It follows $g(z) := f(z)\Phi_n(z)$ is a Mahler function with its associated $A(z)$ and $B(z)$ polynomials containing one fewer factor of $\Phi_n(z)$ each. This may be repeated until either $A(z)$ or $B(z)$ does not contain $\Phi_n(z)$ for any $n$ coprime to $d$. Furthermore, as $\Phi_n(z)$ has integer coefficients, it's easy to see any Mahler number associated with $f(z)$ has the same irrationality exponent as a Mahler number from $g(z)$, thus it's enough to examine $g(z)$. Of course it needs to be checked if the conditions of Lemma \ref{lemma5.4} apply to $g(z)$, however this does not mean the case is lost entirely as our proof of the Lemma suggests. Collating the work in this chapter, we form the following Theorem.

\begin{Theorem}
Let $f(z) \in \Q((z^{-1}))$ be a Mahler function with an associated Mahler number $f(b)$ that satisfies Theorem \ref{thm1.10} and is irrational. Without loss of generality, we may assume $A(z)$ contains no powers of $z$ and there does not exist a $n \in \N$ coprime to $d$ such that $\Phi_n(z)$ divides $A(z)$ and $B(z)$.\\
\\
If $B(z)$ contains non-cyclotomic factors, assume the assumptions of Lemma \ref{lemma4.2} are met. If $B(z)$ contains cyclotomic factors, also assume $A(z)p_{u_n}(z^d) + C(z)q_{u_n}(z^d)$ is eventually coprime to any cyclotomic factor of $B(z)$. Then $\mu(f(b))$ is rational.
\end{Theorem}

\chapter{Conclusion and Future Work}
In this essay, we have managed the problems of computing the irrationality exponent of Mahler numbers that originate from Mahler functions satisfying (\ref{Mahler}) and investigating the rationality of these irrationality exponents, extending the work of Badziahin to a slightly more general form.\\
\\
On the topic of computation, the work of Badziahin extends very well once the new $C(z)$ polynomial is accounted for, and due to this, there only a minimal number of new issues introduced with computing the irrationality exponent of these Mahler numbers. Perhaps the greatest barrier in being able to compute these irrationality exponents is that the method we have currently developed relies on the fact that it is possible to find a big gap in $\Phi(f)$, the set of degrees of denominators of $f(z)$. This task is not always achievable, as there may be no big gaps in $\Phi$ resulting in the irrationality exponent of $f(b)$ being 2, however we have not developed the tools to prove if a certain $\Phi$ has no big gaps. This means in practice, if we cannot find any big gaps in $\Phi$, we cannot conclude what the irrationality exponent of any associated Mahler number may be with any rigor.\\
\\
One way to manage this is to apply Lemma \ref{lemma2.4} and Lemma \ref{lemma2.3} to form an upper bound on the irrationality exponent which approaches 2 as more elements of $\Phi(f)$ are computed. This is done by assuming that a primitive gap of maximal size occurs immediately after the last computed elements of $\Phi(f)$. While this can quickly produce an approximation of the irrationality exponent of a Mahler number as accurate as one desires, which may be sufficient for some purposes, an exact value still eludes us. Investigations into analogs problem have been attempted \cite{bugeaud}, however none have succeeded yet as this turns out to be an extremely challenging problem. Regardless, further investigation would certainly be worthwhile.\\
\\
Another path of further investigation is generalising what Mahler functions we are studying beyond that of (\ref{Mahler}). Recall the original definition of a Mahler function from (\ref{oldMahler}). Finding a method to compute the irrationality exponent of all Mahler numbers rather than the subclass we have focused on would be highly valuable, however it is not likely that the work of Badziahin will extend nearly as well when compared to our work here as we have at least quadratic-style relationships for $f(z)$ instead of the much more simple linear style equations we have restricted to in (\ref{Mahler}).\\
\\
In studying if the irrationality exponent is rational, we succeed in finding several conditions for which the irrationality exponent of an irrational Mahler number is rational, however the $C(z) \neq 0$ assumption does not allow us to solve the problem in full generality, and our assumptions are necessary to make any significant progress. While an obvious area of future work is to work to discard these assumptions or loosen them as much as possible, it should be noted that this may not be possible.\\
\\
There is very little evidence behind the motivation of Adamczewski and Rivoal's claim that the irrationality exponent of all automatic numbers being rational, other than the fact that no irrational irrationality exponents have been observed across all examples investigated. For our purposes, given the work by Badziahin and our work, we do expect the irrationality exponent of an irrational Mahler number satisfying (\ref{Mahler}) to be rational, but it would not be too surprising should this not be the case for general Mahler numbers from (\ref{oldMahler}).


\bibliographystyle{plain}
\bibliography{bibliography.bib}

%
%
%
%

\end{document}